\documentclass[12pt,a4paper,reqno]{amsart}

\usepackage{algorithm}
\usepackage{algpseudocode}
\usepackage{comment}
\usepackage{amsthm}
\usepackage{tikz}
\usepackage{fullpage}
\usepackage{float}
\usepackage{multirow}
\usepackage{amsmath}
\usepackage{amssymb}
\usepackage{hyperref}
\usepackage{amsthm}
\usepackage{amscd}
\usepackage{amsfonts}
\usepackage{array}
\usepackage{url}
\usepackage{mleftright} 
\usepackage{float}
\usepackage{multirow}
\usepackage{hyperref}
\usepackage{ulem}
\usepackage{soul}
\PassOptionsToPackage{lowtilde}{url}

\newtheorem{theorem}{Theorem}[section]

\newtheorem{corollary}[theorem]{Corollary}
\newtheorem{proposition}[theorem]{Proposition}

\newtheorem{example}[theorem]{Example}

\def\1{{\bf 1}}

\def\b{\mbox{\boldmath $b$}}

\def\m{\mbox{\boldmath $m$}}

\DeclareMathOperator{\dist}{dist}
\DeclareMathOperator{\dgr}{dgr}
\DeclareMathOperator{\rank}{rank}
\DeclareMathOperator{\spec}{sp}

\def\Aa{\mbox{\boldmath $A$}}
\def\BB{\mbox{\boldmath $B$}}
\def\I{\mbox{\boldmath $I$}}
\def\J{\mbox{\boldmath $J$}}
\def\Real{\mathbb R}

\usepackage{notation}

\date{}
\begin{document}
\title{The Shannon capacity of graph powers}


\author{Aida Abiad}
\address{Department of Mathematics and Computer Science, Eindhoven University of Technology, Eindhoven, The Netherlands; Department of Mathematics and Data Science of Vrije, Universiteit Brussel, Belgium}
\email{a.abiad.monge@tue.nl}

\author{Cristina Dalf\'o}
\address{Departament de Matem\`atica, Universitat de Lleida, Igualada (Barcelona), Catalonia}
\email{cristina.dalfo@udl.cat}

\author{Miquel \`Angel Fiol}
\address{Departament de Matem\`atiques, Universitat Politècnica de Catalunya, Barcelona, Catalonia; Barcelona Graduate School of Mathematics, Barcelona, Catalonia; Institut de Matem\`atiques de la UPC-BarcelonaTech (IMTech), Barcelona, Catalonia}
\email{miguel.angel.fiol@upc.edu}

\begin{abstract}
For a graph $G$, its $k$-th graph power $G^k$ is constructed by placing an edge between two vertices if they are within distance $k$. 
We consider the problem of deriving upper bounds on the Shannon capacity of graph powers by using spectral graph theory and linear optimization methods. First, we use the so-called ratio-type bound to provide an alternative and spectral proof of a result by Lov\'asz [\textit{IEEE Trans. Inform. Theory}  1979], which states that, for a regular graph, the Hoffman ratio bound on the independence number is also an upper bound on the Lov\'asz theta number and, hence, also on the Shannon capacity. In fact, we show that Lov\'asz' result holds in the more general context of graph powers. Secondly, we derive another bound on the Shannon capacity of graph powers, the so-called rank-type bound, which depends on a new family of polynomials that can be computed by running a simple algorithm. Lastly, we provide several computational experiments that demonstrate the sharpness of the two proposed algebraic bounds. As a byproduct, when these two new algebraic bounds are tight, they can be used to easily derive the exact values of the Lov\'asz theta number (which relies on solving an SDP) and the Shannon capacity (which is not known to be computable) of the corresponding graph power.\\

\end{abstract}

\maketitle

\section{Introduction}

In transmitting information, the vertices of a graph $G$ can represent the different symbols (forming a message) to be transmitted, where two symbols that can be confused in the transmission are joined by an edge. Then, $G$ is called the \textit{confusability graph}. The maximum number of symbols that can be transmitted without error is the independence number of $G$, usually denoted by $\alpha=\alpha(G)$. Let $C\subset V$ be a maximum independent vertex set, $|C|=\alpha$. Then, in the language of information theory, we say that the \textit{information rate} (assuming that all the symbols are transmitted with the same probability $p_i=1/\alpha$) is 
$$
\sum_{i\in C}p_i\log_2 \frac{1}{p_i}=\log_2 \alpha(G).
$$
Shannon proved that the information rate can be increased using larger strings instead of single symbols.
More precisely, using strings of length $k$, the \textit{$k$-th confusability graph} is the strong product of $G$ by itself $k$ times and is denoted by $G^{\boxtimes k}$.
Thus, by using strings of $k$ symbols, the information rate \textit{per symbol}
is 
$$
\frac{\log_2 \alpha(G^{\boxtimes k})}{k}=\log_2\sqrt[k]{\alpha(G^{\boxtimes k})}.
$$
This led Shannon~\cite{s56} to define the \textit{zero-error capacity} of a graph $G$ as
$$
\Theta(G):=\sup_{n}\alpha(G^{\boxtimes n})^{1/n}=\lim_{n\to\infty}\alpha(G^{\boxtimes n})^{1/n}.
$$
This parameter $\Theta(G)$ (known as \textit{Shannon capacity}) measures the zero-error capacity of a communication channel that uses symbols represented by the graph's vertices. 
Since its introduction by Shannon \cite{s56}, it has received much attention in the literature. However, although Shannon defined it in 1956, very little is known about it. In fact, the computational complexity of the Shannon capacity is unknown, and the same happens for its value in some small graphs: while Lov\'asz \cite{l79} showed that  $\Theta(C_5)=\sqrt{5}$, the problem of determining the value of $\Theta(C_7)$ remains wide open. Some limits on the independence number of cycles of odd length using the Shannon capacity were shown by Bohman \cite{b03,b05}. The Shannon capacity exhibits some bizarre behavior. For instance,  Alon \cite{a98} showed that there are graphs $G$ and $H$ such that the Shannon capacity of their disjoint union $G\cup H$ is (much) larger than
the sum of their Shannon capacities, disproving a conjecture of Shannon (1956). Moreover, Alon and Lubetzky \cite{al06} proved that ``the series of independence numbers in strong powers of a fixed graph can exhibit a complex structure, implying that the Shannon capacity of a graph cannot be approximated (up to a subpolynomial factor of the number of vertices) by any arbitrarily large (but fixed) prefix of the
series''.


%

Two upper bounds on the Shannon capacity $\Theta(G)$ are relevant for our work. 
The first is the Lov\'asz theta number (or theta function) $\vartheta(G)$, which can be computed as the solution to a semidefinite program dealing with arrangements of vectors associated with
$G$, and thus, can be computed in polynomial time. Introduced by Lov\'asz~\cite{l79}, the theta function was used to verify that $\Theta(C_5) = \vartheta(C_5) = \sqrt{5}$. The second upper bound on $\Theta(G)$, denoted as $R(\BB)$, was proposed by Haemers \cite{H1978,h79} and uses the rank of some particular matrices on a field $\mathbb{F}$ associated with the graph $G$. Haemers~\cite{h79} introduced this bound $R(\BB)$ to provide negative answers to three questions put forward by Lov\'asz in~\cite{l79}. 


In this paper, we consider the Shannon capacity of graph powers. For a positive integer $k$, the $k$-th \textit{power of a graph} $G =(V, E)$, denoted by $G^k$, is the graph with vertex set $V$ in which two distinct vertices are joined by an edge if there is a path in $G$ of length at most $k$ between them. 
Graph powers help design efficient algorithms for specific combinatorial optimization problems; see, for instance, 
Andreae and Bandelt \cite{2}, and Bender and Chekuri \cite{bc00}. In distributed computing, the $k$-th power of a graph $G$ represents the possible flow of information during $k$ rounds of communication in a distributed network of processors organized according to a graph $G$; see Linial \cite{l92}.


The investigation of the behaviour of various parameters of powers of a fixed graph leads to many fascinating problems, some of which are motivated by questions in information theory, communication complexity, geometry and Ramsey theory, see \cite{Apowers,powersthesis} for an overview. For some results of combinatorial nature involving graph powers, see, for example, Alon and Mohar \cite{AM2002}, Das and Guo \cite{dg13}, Hajiabolhassan \cite{13}, Kearney and Corneil \cite{15}, and Lau and Corneil \cite{16}, Devos, McDonald, and Scheide~\cite{Devos2013AveragePowers}), or Basavaraju, Chandran, Rajendraprasad, and Ramaswamy~\cite{bscrr14}. Recently, a particular case of the graph power has been investigated in the context of the Shannon capacity by de Boer, Buys, and Zuiddam \cite{BBZ2024} and by Buys, Polak, and Zuiddam\cite{BPZ2025}. Note that by graph power, some authors mean a graph in which two distinct vertices of $G^k$ are joined by an edge if there is a path in $G$ of length exactly $k$ between them, which is a variation of our definition.

In this paper we use the independence number of graph powers to provide new algebraic upper bounds on the Shannon capacity of $G^k$.
The distance generalization of the independence number, known as the $k$-\textit{independence number} of a graph $G$ and denoted by $\alpha_k(G)$, is the maximum number of vertices in $G$ at distance pairwise at least $k$, or alternatively, $\alpha_k(G)=\alpha(G^k)$. While the latter parameter relation holds, the spectrum of $G^k$ cannot be in general deduced from the spectrum of $G$, see for instance Das and Guo~\cite{dg13}, and Abiad, Coutinho,  Fiol, Nogueira, and Zeijlemaker \cite[Section 2]{acfnz21}. Our bounds on the Shannon capacity of the graph power $G^k$ are in terms of the eigenvalues of the original graph $G$. We derive them by using some recent eigenvalue bounds on the $k$-independence number of a graph. Indeed, since the distance generalization of the independence number, $\alpha_k(G)$, is computationally difficult to determine, lot of attention has been put on deriving bounds for it which can be efficiently computed, see e.g. \cite{AZ2024,FH1997,F1999,lw2021,Taoqiu2019SharpGraphs,AF2004}. One of such bounds for $\alpha_k(G)$ was derived by Abiad, Coutinho, and Fiol \cite{acf19}. It uses the eigenvalues of a graph $G$ and a polynomial of degree at most $k$.  
\begin{theorem}[Ratio-type bound \cite{acf19}]
\label{thm:acf19}
   Let $G=(V,E)$ be a regular graph with $n$ vertices and adjacency eigenvalues $\lambda_1 \geq \cdots \geq \lambda_n$. 
    Let $[2, n] = \{2, 3,\ldots, n\}$. Given a polynomial $p \in \mathbb{R}_k[x]$,  consider the parameters $W(p) = \max_{u\in V} {(p(A))_{uu}}$ and $\lambda(p) = \min_ {i\in[2,n]}{p(\lambda_i)}$, and assume $p(\lambda_1) > \lambda(p)$. Then, the $k$-independence number of a graph satisfies
    $$
\alpha_k(G) \leq n \frac{W(p) - \lambda(p)}{p(\lambda_1) - \lambda(p)}.
$$
\end{theorem}

For $k=1$, the above bound gives the celebrated Hoffman \textit{ratio bound} on the independence number of a regular graph (unpublished; see, for instance, Haemers \cite{h95}):
\begin{equation}\label{eq:ratiobound}
\alpha(G) \leq n\frac{-\lambda_n }{\lambda_1 - \lambda_n}.
\end{equation}
Fiol \cite{f20} optimized the ratio-type bound for general $k$ using linear programming methods and the so-called minor polynomials. The optimal polynomials for the ratio-type bound were explicitly derived by Abiad, Coutinho, and Fiol \cite{acf19} for $k=2$, and by Kavi and Newman \cite{kn23} for $k=3$. The ratio-type bound, which will be the main character of the first part of this paper, has recently witnessed applications in coding theory, improving the state-of-the-art bounds on the maximal size of code for some metrics, see \cite{ANR2025,AKR2024,AAR2025}. Here, we investigate its application to information theory.

Fiol \cite[Section 4.1]{f20} showed that, for partially walk-regular graphs, the aforementioned ratio-type bound is in fact an upper bound on the Shannon capacity of $G^k$. Continuing this line of research, in the first part of this paper we investigate multiple consequences of this upper bound on the Shannon capacity of the graph power. Indeed, we use the ratio-bound to show a (non-trivial) generalization to graph powers of the result by Lov\'asz \cite{l79} that states that the Lov\'asz theta number upper bounds the ratio bound (non-trivial because the spectra of $G$ and $G^k$ are generally not related), see Theorem \ref{thm:ratiotypeboundsupperboundslovasztheta}. This is useful because the ratio-type bound can be computed efficiently by using a linear programming (LP) and the minor polynomials, and, when our bound is tight, as a byproduct,  we obtain the value of the Lov\'asz theta number (which is a more computationally expensive bound to compute since it depends on solving a semidefinite program) and of the Shannon capacity of the power graph, see Corollary \ref{coro:firstpart}. Moreover, our approach has the advantage of yielding closed formulas for the new eigenvalue bounds on the Shannon capacity and, thus, allowing asymptotic analysis. 



Furthermore, in the second part of this paper, we present another algebraic bound on the Shannon capacity of graph powers; the so-called rank-type bound. In order to derive and compute such bound we  introduce a new class of polynomials (the Shannon polynomials), which allow us to extend the well-known rank bound by Haemers \cite{h79} on the Shannon capacity, see Theorem \ref{th:Hamererankextension2} and Theorem \ref{thm:secondboundShannonpolys}. We also provide implementations of the new rank-type bound for $G^k$, formulating it both as a linear optimization problem (LP) and as an algorithm.

Finally, we run several computational experiments to illustrate the performance of the two newly proposed bounds on the Shannon capacity of graph powers: the ratio-type bound and the rank-type bound. Both bounds are tight for several instances of tested graphs. Notably, the ratio-type bound often performs as well as the celebrated Lov\'asz theta  bound (see Tables  	\ref{fig:alsoLovaszthetak=2}-\ref{table:cyclepowersranktypeandratiotypeboundsk=5}), but our spectral bound is significantly easier to compute than the  Lov\'asz theta bound. Indeed, the ratio-type eigenvalue bound (as well as the rank-type bound) uses an LP and is, for small graphs like the ones we tested,
significantly faster than solving the SDP needed to compute the the Lov\'asz theta number. In general, both the ratio-type and the rank-type proposed bounds seems to be incomparable.

The rest of this paper is structured as follows. In Section \ref{sec:prelim}, we introduce the necessary definitions and standardized notation. In Section \ref{bounds-Theta}, we start by illustrating how the ratio-type bound, which can be optimized using the minor polynomials, gives an upper bound on the Shannon capacity of graph powers, and we investigate multiple consequences of this result for small powers. Next, we show that this result provides an alternative proof of the fact that the ratio bound is also an upper bound on the Shannon capacity of a graph, a result that was originally shown by Lov\'asz \cite[Theorem 9]{l79} by using a different approach. In Section \ref{sec:optimizingboundsalphak}, we show an upper bound on the Shannon capacity of graph powers, the rank-type bound, which uses polynomials and the minimum rank of a certain matrix, and which extends the celebrated Haemers' rank bound \cite{h79}. To do so, we introduce a new family of polynomials (called Shannon polynomials), and we show that they can be computed both by solving a linear program and by running a simple algorithm. Finally, in the Appendix, we present several computational results for the two investigated algebraic upper bounds on the Shannon capacity of graph powers (the ratio-type bound and the rank-type bound), which seem to indicate that the bounds are incomparable also for general $k$.

\section{Preliminaries}
\label{sec:prelim}

Throughout this paper, $G=(V,E)$ denotes a (simple and connected) graph with $n=|V|$ vertices, $m=|E|$ edges, and adjacency matrix $\Aa=\Aa(G)$ with spectrum
$\spec G=\{\theta_0^{m_0},\theta_1^{m_1},\ldots,\theta_d^{m_d}\},$
where the different eigenvalues are in decreasing order, $\theta_0>\theta_1>\cdots>\theta_d$, and the superscripts stand for their multiplicities (since $G$ is supposed to be connected, $m_0=1$). We denote by $\mathbb{R}_k[x]$ the collection of polynomials in $x$ over $\mathbb{R}$ of degree at most $k \in \mathbb{N}_{\geq 0}$.

Let us now recall some known concepts. A \textit{closed walk} in a graph is a
walk from a vertex back to itself. 
A graph $G$ is called \textit{walk-regular} if the number of closed walks of any length from a vertex to itself does not depend on the choice of the vertex, a concept introduced by Godsil and McKay in \cite{gm80}.
Recall that the \textit{girth} of a graph $G$ is the length of a shortest cycle contained in $G$. Then, a natural generalization of a walk-regular graph is as follows: A graph $G$ is called \textit{$k$-partially walk-regular}, for some integer $k\ge 0$, if the number of closed walks of a given length $l\le k$, rooted at a vertex $v$, only depends on $l$. Thus, every (simple) graph is $k$-partially walk-regular for $k=0,1$, and every regular graph with girth $g$ is $k$-partially walk-regular with $k=2\lfloor (g-1)/2\rfloor$ (since all walks of length $\ell\le 2k$ are paths going forward and backward through some $\ell$ edges).

A graph $G$ is called \textit{distance-regular} if, for any two vertices $u$ and $v$ at distance $\dist(u,v)=l$, the number of vertices at distance $i$ from $u$ and at distance $j$ from $v$, denoted by $p_{ij}^l$, depends only upon $l$, $i$, and $j$. As another characterization, $G$ is distance-regular if the number of $l$-walks between two vertices $u,v$ only depends on their distance $\dist(u,v)$.
A graph $G$ is called \textit{$k$-partially distance-regular} if it is distance-regular up to distance $i\leq k$.
Then, every $k$-partially distance-regular graph is $2k$-partially walk-regular. Moreover, $G$ is $k$-partially walk-regular for any $k$ if and only if $G$ is walk-regular.  For example, it is well known that every distance-regular graph is walk-regular (but the converse does not hold).
In other words, if $G$ is a $k$-partially walk-regular graph, for any polynomial $p\in \Real_k[x]$, the diagonal of $p(\Aa)$ is constant with entries
$$
(p(\Aa))_{uu}=\frac{1}{n}\tr\big(p(\Aa) \big)=\frac{1}{n}\sum_{i=1}^n  p(\lambda_i)\quad \mbox{for all $u\in V$}.
$$

\section{Bounds using the minor polynomials}
\label{bounds-Theta}

\subsection{Ratio-type bound}

In this section, we assume that the graph $G$ is regular.
When, moreover, $G$ is $k$-partially walk-regular for some $k\ge 1$, an upper bound for the Shannon capacity of the graph power, $\Theta(G^k)$, was provided by Fiol \cite[Proposition 4.2]{f20}.
This bound uses the so-called minor polynomials, which are defined as follows.  
Let $G$ have spectrum $\spec G=\{\theta_0,\theta_1^{m_1},\ldots,\theta_d^{m_d}\}$ with $\theta_0>\theta_1>\cdots>\theta_d$. 
The \textit{$k$-minor polynomial} $f_k\in \Real_k[x]$ is the polynomial defined by $f_k(\theta_0)=1$ and $f_k(\theta_i)=x_i$, for $1\le i\le d$, where the vector $(x_1,x_2,\ldots,x_d)$ is a solution of the following linear programming:
\begin{equation}
\boxed{
 \begin{array}{rl}
 & \\
 {\tt minimize} & \displaystyle\sum_{i=0}^d m_i f(\theta_i)\\
 {\tt with\ constraints} & f(\theta_i)=x_i,\\
                        & f[\theta_0,\ldots,\theta_m]=0,\quad \mbox{for } m=k+1,\ldots,d\\
                       & x_i\ge 0, \quad \mbox{for } i=1,\ldots,d.\\
                        &
 \end{array}
}
 \label{LP-minor-f}
 \end{equation}
Here, $f[\theta_0,\ldots,\theta_m]$ denotes the $m$-th divided differences of Newton interpolation, recursively defined by  $f[\theta_i,\ldots,\theta_j]=\frac{f[\theta_{i+1},\ldots,\theta_j]-f[\theta_i,\ldots,\theta_{j-1}]}
{\theta_j-\theta_{i}}$, where $j>i$, starting with $f[\theta_i]=f_k(\theta_i)=x_i$, for $0\le i\le d$.
Fiol \cite{f20} showed that $f_k$ has degree $k$, with exactly $k$ zeros in the mesh $\theta_1,\ldots,\theta_d$. For some values of $k$, there exist explicit formulas for the corresponding $k$-minor polynomial:
\begin{description}
\item[\boldmath{$k=0$}]
The $0$-minor polynomial is a constant, $f_0(x)=1$.
\item[\boldmath{$k=1$}]
The $1$-minor polynomial is $f_1(x)=\frac{x-\theta_d}{\theta_0-\theta_d}$.
\item[\boldmath{$k=2$}] 
The $2$-minor polynomial has zeros at consecutive eigenvalues $\theta_{i}$ and $\theta_{i-1}$, where $\theta_i$ is the largest eigenvalue not greater than $-1$. Then, Abiad, Coutinho, and Fiol \cite{acf19} showed that the 2-minor polynomial is $$f_2(x)=\frac{(x-\theta_i)(x-\theta_{i-1})}{(\theta_0-\theta_{i})(\theta_0-\theta_{i-1})}.$$
\item[\boldmath{$k=3$}]
A $3$-minor polynomial was implicitly given by Kavi and Newman \cite[Theorem 11]{kn23} as follows:  With $d\ge 3$, let $\theta_s$ be the smallest  eigenvalue such that $\theta_s\ge-\frac{\theta_0^2+\delta\theta_d-\Delta}{\delta(\theta_d+1)}$ with $\Delta=\max_{u\in V}\{(\Aa^3)_{uu}\}$. Consider the polynomial $p_3(x)=x^3+bx^2+cx$ with $b=-(\theta_s+\theta_{s+1}+\theta_d)$ and $c=\theta_d\theta_s+\theta_d\theta_{s+1}+\theta_s\theta_{s+1}$. Then, a $3$-minor polynomial is $$f_3(x)=\frac{p_3(x)-p_3(\theta_d)}{p_3(\theta_0)-p_3(\theta_d)}=\frac{(x-\theta_s)(x-\theta_{s+1})(x-\theta_d)}{(\theta_0-\theta_s)(\theta_0-\theta_{s+1})(\theta_0-\theta_d)}.$$
\item[\boldmath{$k=d$}] 
The $d$-minor polynomial is $f_d(x)=\frac{1}{n}H(x)$, where $H(x)$ is the Hoffman polynomial with zeros at $\theta_i$, for $i\neq 0$, and $H(\theta_0)=n$. 
\end{description}
In general, the $k$-minor polynomial for a given graph is not unique, but always, if $k$ is odd, it has
zeros at $\theta_d$ and $(k-1)/2$ pairs of consecutive eigenvalues $(\theta_i,\theta_{i+1})$. Analogously, if $k$ is even, a $k$-minor polynomial has
zeros at $k/2$ pairs of consecutive eigenvalues.
\begin{figure}[t]
	\begin{center}
 \includegraphics[width=8cm]{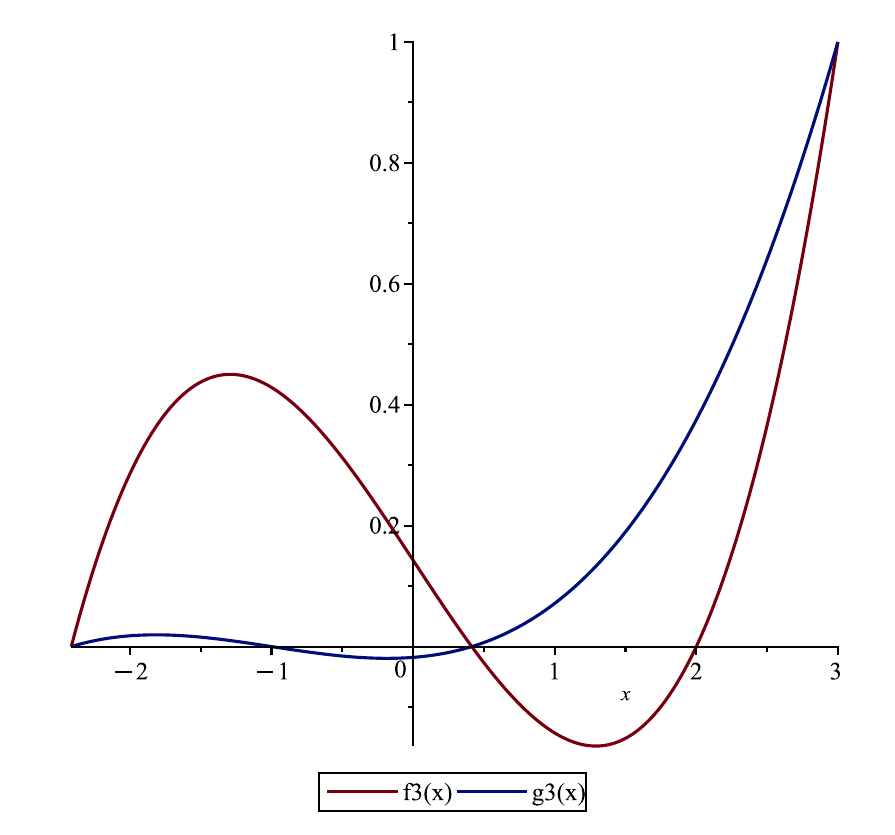}
	\end{center}
	\caption{Two $3$-minor polynomials for the Coxeter graph in Example \ref{ex:f3+g3}.}
	\label{fig:f3+g3}
\end{figure}
\begin{example}
\label{ex:f3+g3}
The Coxeter graph $C$ is a cubic distance-regular graph with $n=28$ vertices, diameter $d=4$, and spectrum $\spec C=\{3^1,2^8,(\sqrt{2}-1)^6,-1^7, (-\sqrt{2}-1)^6\}$ (see, for instance, Biggs \cite{biggs}). Then, the two following polynomials are $3$-minor polynomials (see Figure \ref{fig:f3+g3}):
\begin{align*}
f_3(x) &=\frac{1}{14}(x^3-5x+2),\\
g_3(x) &= \frac{1}{56}(x^3+3x^2+x-1).
\end{align*}
In addition to $\theta_d=\theta_4=-\sqrt{2}-1$, the polynomial $f_3(x)$ has consecutive zeros $\theta_s=\theta_1=2$ and $\theta_2=\sqrt{2}-1$, while $g_3(x)$ has consecutive zeros $\theta_s=\theta_2=\sqrt{2}-1$ (this is the value given by the condition 
in {\boldmath{$k=3$}} by Kavi and Newman \cite{kn23}) and $\theta_3=-1$.
Moreover, $\tr f_3(\Aa)=\tr g_3(\Aa)=4$.
\end{example}

 For $k$-partially walk-regular graphs, Fiol \cite{f20} observed that the ratio-type bound from Theorem \ref{thm:acf19} can be written in the following form. We still call the following bound ratio-type bound (but note that in this case the bound only applies for the case of partially walk-regular graphs).

\begin{theorem}[\cite{f20}]
\label{th:Theta-minor-pol} 
Let $G$ be a regular $k$-partially walk-regular graph with $n$ vertices, adjacency matrix $\Aa$, and spectrum $\spec G=\{\theta_{0}^{m_0}, \ldots, \theta_{d}^{m_d}\}$. Let $f_k\in \mathbb{R}_k[x]$ be a $k$-minor polynomial. Then, for every $k=0,\ldots, d-1$, the Shannon capacity of $G^k$ satisfies
	\begin{equation}\label{boundFiolminor}
	 \Theta(G^k)\leq \tr f_k(\Aa)=\sum_{i=0}^{d}m_if_k(\theta_i).
	\end{equation}
\end{theorem}
\medskip

Since every regular graph is 1- and 2-partially walk-regular, we have the following consequence for small graph powers.
\begin{corollary}
\label{coro:Sha-Hoff}
Let $G$ be a $\delta$-regular graph on $n$ vertices, with spectrum $\spec G=\{\theta_{0}^{m_0},$ $ \theta_1^{m_1},\ldots, \theta_{d}^{m_d}\}$, where $\theta_0(=\delta)>\theta_1>\cdots > \theta_d$. Then, we have the following upper bound for the Shannon capacity of the power graphs $G^k$ for $k=1,2,3$.  
\begin{itemize}
    \item[$(i)$]
    If $H_1$ is the Hoffman ratio bound (Eq. \eqref{eq:ratiobound}), then 
    \begin{equation}
    \Theta(G) \leq H_1=n\frac{-\theta_d}{\delta-\theta_d}.
    \label{eq:Sha-Hoff1}
   \end{equation}
\item[$(ii)$]
If $\delta\ge 2$ and $\theta_i$ is the largest eigenvalue not greater than $-1$, then 
\begin{equation}
\Theta(G^2) \leq H_2=n\frac{\delta+\theta_i\theta_{i-1}}{(\delta-\theta_i)(\delta-\theta_{i-1})}.
\label{eq:Sha-G2}
\end{equation}
\item[$(iii)$]
If $G$ is at least $3$-partially walk-regular with $\delta\ge 3$, $\theta_s$ is the smallest  eigenvalue such that $\theta_s\ge-\frac{\delta^2+\delta\theta_d-n_t}{\delta(\theta_d+1)}$, with $n_t$ being the common number of triangles rooted at every vertex, then 
\begin{equation}
\Theta(G^3) \le H_3=
n\frac{2n_t-\delta(\theta_s+\theta_{s+1}+\theta_d)-\theta_s\theta_{s+1}\theta_d}{(\delta-\theta_s)(\delta-\theta_{s+1})(\delta-\theta_d)}.
\label{eq:Sha-G3}
\end{equation}
\end{itemize}
\end{corollary}
\begin{proof}
As mentioned above, $f_1(x)=\frac{x-\theta_d}{\theta_0-\theta_d}$ and  $f_2(x)=\frac{(x-\theta_i)(x-\theta_{i-1})}{(\theta_0-\theta_{i})(\theta_0-\theta_{i-1})}$. Then, the result follows from Theorem \ref{th:Theta-minor-pol} since, by using that 
$\sum_{i=0}^d m_i=n$, $\sum_{i=0}^dm_i\theta_i=0$, and 
$\sum_{i=0}^dm_i\theta_i^2=n\delta$, 
we obtain
\begin{align*}
\tr f_1(\Aa) &=\frac{1}{\delta-\theta_d}\sum_{i=0}^{d} m_i(\theta_i-\theta_d)=\frac{1}{\delta-\theta_d}\left(\sum_{i=0}^{d}m_i\theta_i-\theta_d \sum_{i=0}^{d}m_i\right)
= n\frac{-\theta_d}{\delta-\theta_d}=H. \\
 \tr f_2(\Aa)& =\frac{1}{(\delta-\theta_i)(\delta-\theta_{i-1})}\sum_{j=0}^{d} m_j(\theta_j-\theta_i)(\theta_j-\theta_{i-1})\\
  &=\frac{1}{(\delta-\theta_i)(\delta-\theta_{i-1})}\left(\sum_{j=0}^{d}m_j\theta_j^2-(\theta_i+\theta_{i-1})\sum_{j=0}^{d}m_j\theta_j+\theta_i\theta_{i-1}\sum_{j=0}^{d}m_j\right)\\
 & = n\frac{\delta+\theta_i\theta_{i-1}}{(\delta-\theta_i)(\delta-\theta_{i-1})}.
\end{align*}
Finally, if $G$ is a 3-partially walk-regular graph, then
$\Delta=\frac{1}{n}\sum_{i=0}^dm_i\theta_i^3=2n_t$, and the value of $\tr f_{i}(\Aa)$ is computed similarly
(see, again, Kavi and Newman \cite[Thm. 11]{kn23}).
\end{proof}

We should note that \eqref{eq:Sha-Hoff1} was also shown by Lov\'asz in \cite[Thms. 1 and 9]{l79}.

\begin{example} 
The cycle $C_5$ on $n=5$ vertices has maximum and minimum eigenvalues $\theta_0=2$ and $\theta_2=-(1+\sqrt{5})/2$. Then, \eqref{eq:Sha-Hoff1} gives $\Theta(C_5)\le \sqrt{5}$. This is precisely the inequality proved by Lov\'asz \cite{l79} by using his $\vartheta$ function to conclude that $\Theta(C_5)=\sqrt{5}$. 
Thus, in some cases, Corollary \ref{coro:Sha-Hoff} provides an alternative way of computing the Shannon capacity of a graph. We should also note that the proof of Theorem \ref{th:Theta-minor-pol}, which appears in Fiol \cite{f20}, only uses eigenvalue interlacing as a main technique.
\end{example}
Since it holds that $\alpha(G)\le\Theta(G)$, we also have the following corollary.
\begin{corollary}
\label{coro:Shannon=Hoffman}
Let $G$ be a regular graph with independence number $\alpha_i=\alpha(G)$ for some $i=1,2,3$.
If $\alpha_i$ equals the ratio bound  $H_i$ in \eqref{eq:Sha-Hoff1}-\eqref{eq:Sha-G3}, $\alpha_i=H_i$, then $\Theta(G^i)=H_i$.
\end{corollary}
\begin{figure}[t]
	\begin{center}
 \includegraphics[width=8cm]{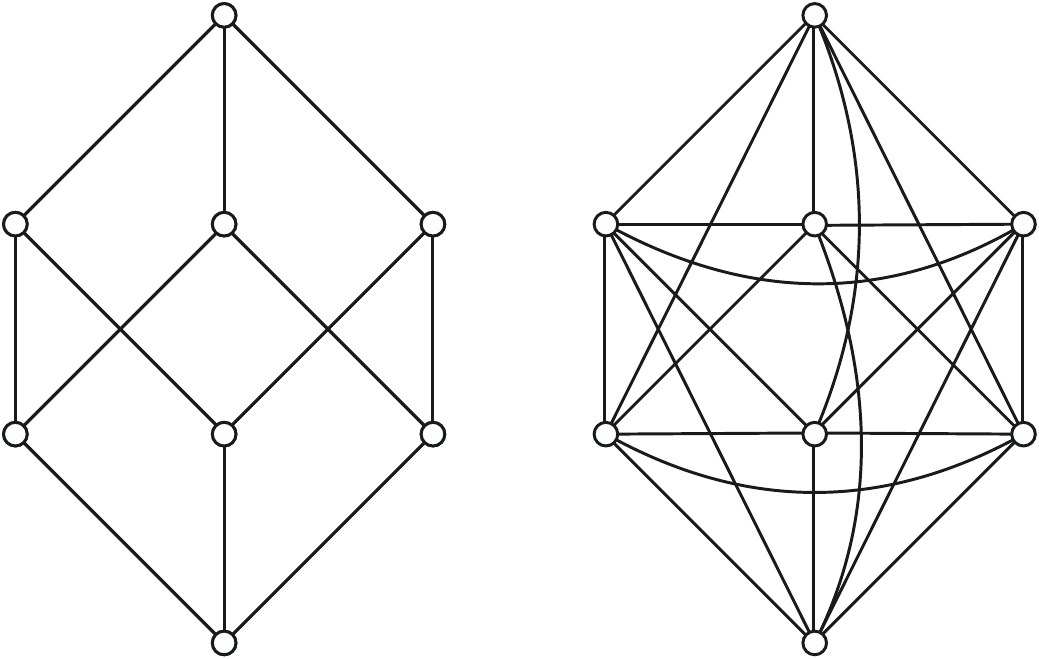}
	\end{center}
	\vskip-.25cm
	\caption{The cube $G=Q_3$ and its 2-power graph $G^{2}=Q_3^{2}$, which is a $(8,6;4,6)$-strongly regular graph  with spectrum $\spec Q_3^{2}=\{6,0^4,-2^3\}$.}
	\label{fig:cub}
\end{figure}
Let us see one example of each case.
\begin{example}
\label{ex:Qd}
Let $G=Q_d$ be the $d$-cube graph. Then, its $(d-1)$-power graph $G^{d-1}=Q_d^{d-1}$ has maximum and minimum eigenvalues $\theta_0=2^d-2$ and $\theta_d=-2$. Then, the Hoffman bound is $H_1=2$, which coincides with its independence number (since $Q_d$ is a $2$-antipodal distance-regular graph). Then, we conclude that $\Theta(Q_d^{d-1})=2$. See the example of $G=Q_3$ and $G^{2}=Q_3^{2}$ in Figure \ref{fig:cub}.
\end{example}

\begin{example}
\label{ex:coxeter2}
Considering again the Coxeter graph $C$  with $n=28$ vertices, degree $\delta=3$, and spectrum
$\spec C=\{3,2^8,(-1+\sqrt{2})^6,-1^7,(-1-\sqrt{2})^6\}$, we have the following: 
\begin{itemize}
\item
The largest eigenvalue not greater than $-1$ is $\theta_{3}=-1$.
 Thus, using this value together with $\theta_{4}=-1-\sqrt{2}$ in  \eqref{eq:Sha-G2}, we get that $H_2=7$. Thus, since it can be checked that $\alpha_2=7$, we have that $\Theta(C^2)=7$.
 \item 
 Since the number of triangles rooted at every vertex is $n_t=0$, the smallest eigenvalue at least $-\frac{\delta-\theta_4}{\theta_4+1}=\sqrt{2}-1$ is $\theta_2=\sqrt{2}-1$. Then, using this value together with $\theta_3=-1$ in  \eqref{eq:Sha-G3}, we have that $H_3=4$ (that is, the value of $\tr g_3(\Aa)$ in Example \ref{ex:f3+g3}).
 Since it can be checked that $\alpha_3=4$, we get
 $\Theta(G^3)=4$.
\end{itemize}
\end{example}

\begin{table}[!ht]
	\begin{center}
		\begin{tabular}{|c|ccccc|cc|c|}
			\hline
			$k$ & $x_4$  & $x_3$ & $x_2$ & $x_1$ & $x_0$ & $\tr f_k$ & $\alpha_k$ & $\Theta(G^k)$\\
			\hline\hline
			$1$ & 0 & 0.2612 & 0.5224 & 0.8153 & 1 & 12.4852 & 12 & 12.4852 \\
			\hline
			$2$ & 0.3867 & 0 & 0 & 0.4599 & 1 & 7 & 7 & 7 \\
			\hline
			$3$  & 0  & 0 & 0 & 0.375 & 1 & 4 & 4 & 4 \\
			\hline
			$4$  & 0 & 0 & 0  & 0 & 1 & 1 & 1 & 1 \\
			\hline
		\end{tabular}
	\end{center}
	\caption{Values $x_i=f_k(\theta_i)$ of the $k$-minor polynomials of the Coxeter graph $C$, with the ratio bounds $\tr f_k$, $k$-independence numbers $\alpha_k$, and Shannon capacities $\Theta(G^k)$ for $k=1,2,3,4$.}
	\label{table2}
\end{table}

	\begin{figure}[t]
		\begin{center}
			\includegraphics[width=8cm]{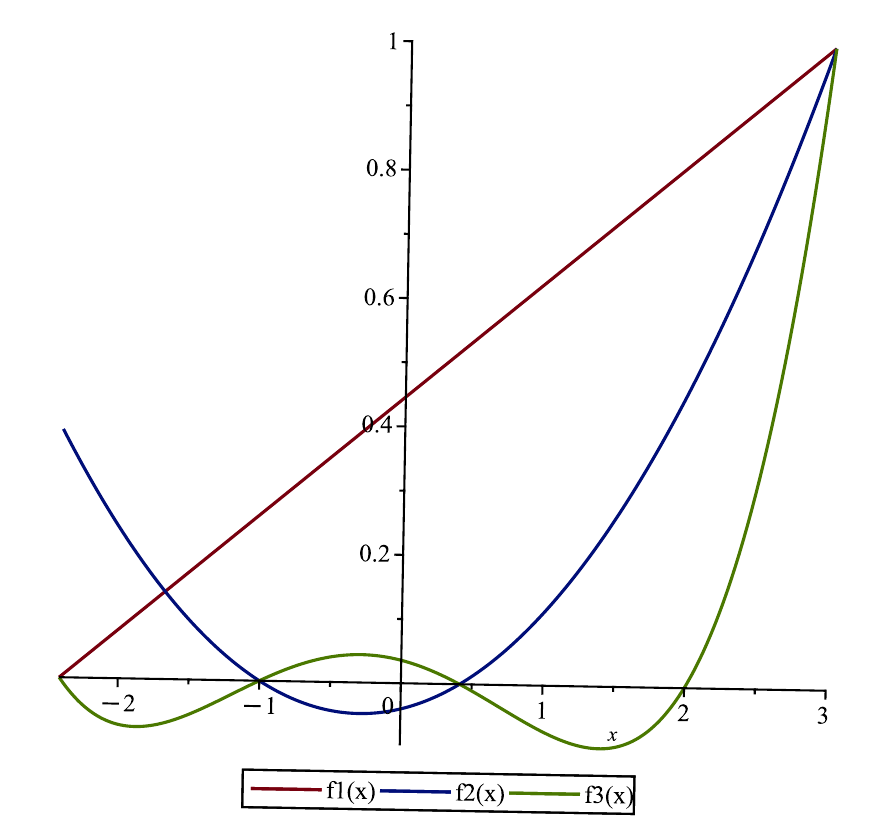}
			\caption{The $\{1,2,4\}$-minor polynomials of the Coxeter graph.}
			\label{fig1}
		\end{center}
	\end{figure}

In fact, Example \ref{ex:Qd} is a particular case of a more general phenomenon, which we show in the following result. First, recall that a graph $G=(V,E)$ of diameter $d$ is called \textit{antipodal} if there exists a partition of
the vertex set into classes with the property that two distinct
vertices are in the same class if and only if they are at distance $d$ (see, for instance, Godsil \cite{g93}.) If all the classes have the same cardinality, say $r$, we say that $G$ is an $r$-\textit{antipodal graph}. Moreover, 
a \textit{strongly regular} graph with parameters $(n, k; a, c)$ is a regular graph $G = (V, E)$ with $n$ vertices and degree $k$ such that for some given integers $a,c\geq0$,
every two adjacent vertices have $a$ common neighbors, and
every two non-adjacent vertices have $c$ common neighbors.
\begin{proposition}
Let $G$ be an $r$-antipodal distance-regular graph with $n$ vertices and diameter $d$. 
Then, the power graph $G^{d-1}$ is an $(n,k;a,c)$-strongly regular graph with parameters $k=c=n-r$ and $a=n-2r$, and it holds that the Shannon capacity is $\Theta(G^{d-1})=r$.
\end{proposition}
\begin{proof}
We know that every set of vertices at a distance $d$ from each other has cardinality $r$. Then, the adjacency matrix $\Aa$ of $G^{d-1}$ is equal to the Kronecker product $\J_r\otimes\Aa_{n/r}$, where $\J_r$ is the all-1 matrix of dimension $r$, and $\Aa_{n/r}$ is the adjacency matrix of the complete graph $K_{n/r}$. 
Since $\spec \J=\{r,0^{r-1}\}$ and $\spec \Aa_{n/r}=\{n/r-1,-1^{n/r-1}\}$, we have 
$$
\spec \Aa=\{n-r,0^{n-\frac{n}{r}}, -r^{\frac{n}{r}-1}\},
$$ 
that is, the product of both spectra. It is well known that a (connected) regular graph with three distinct eigenvalues $\theta_0>\theta_1>\theta_2$ is an $(n,k;a,c)$-strongly regular graph with parameters satisfying $k=\theta_0$, $c-k=\theta_1\theta_2$, and $a-c=\theta_1+\theta_2$ (see, for instance, Godsil \cite{g93}). In our case, with $\theta_0=n-r$, $\theta_1=0$, and $\theta_2=-r$, we get $k=n-r$, $c=n-r$, and $a-(n-r)=-r$, then $a=n-2r$, as claimed.
Moreover, the Hoffman bound for $G^{d-1}$ is
$$
H_1=n\frac{r}{(n-r)+r}=r=\alpha(G^{d-1}),
$$
and, by Corollary \ref{coro:Shannon=Hoffman}, we have $\Theta(G^{d-1})=r$.
\end{proof}

\subsection{Lov\'asz theta number}

As mentioned in the introduction, an upper bound for the Shannon capacity of a graph $G$ is given by the Lov\'asz theta number, denoted by $\vartheta(G)$, which can be computed using semidefinite programming (SDP) methods. Two possible definitions of this parameter are as follows.
Let $G=(V,E)$ be a graph on $n$ vertices.
Let $\Aa=(a_{uv})$ range over all $n\times n$ symmetric matrices such that $a_{uv}= 1$ when $u=v$ or $uv\not\in E$. Let $\rho(\Aa)$ be the spectral radius of $\Aa$. Then, the \textit{Lov\'asz theta number} of $G$ is
\begin{equation}
\label{theta1}
\vartheta(G)=\min_{\textit{\textbf{A}}} \rho(\Aa). 
\end{equation}
Alternatively, the dual method to this is 
\begin{equation}
\vartheta(G)=\max_{\textit{\textbf{B}}} \tr(\BB\J),
\label{theta2}
\end{equation}
where $\BB$ ranges over all $n\times n$ symmetric positive semidefinite matrices such that $b_{uv}=0$ for every $uv\in E$ and $tr \BB=1$,
and $\J$ is the all-1 matrix
(see Lov\'asz \cite{l79}).

Thus, in the case of power graphs, we have
 $$
 \alpha_k(G)=\alpha(G^k)\le \Theta(G^k)\le \vartheta(G^k).
 $$


In fact, as it happened with the Shannon capacity of a graph, for $k$-partially walk-regular graphs, the ratio-type upper bound from Theorem \ref{th:Theta-minor-pol} also upper bounds the Lov\'asz theta number of the power graph:
\begin{theorem}\label{thm:ratiotypeboundsupperboundslovasztheta}
Let $G$ be a $k$-partially walk-regular graph as in Theorem \ref{th:Theta-minor-pol}. Let $f_k$ be a minor polynomial of $G$ for a given $k=1,\ldots, d$. Then,
\begin{equation}
\label{thetaLovasz-fk}
\vartheta(G^k)\le \tr f_k(\Aa)=\sum_{i=0}^{d}m_if_k(\theta_i).
\end{equation}
\end{theorem}
\begin{proof}
Consider the matrix $\Aa_k= \J-nf_k(\Aa)+\Psi(f_k)\I$, where $\Psi(f_k)=\tr f_k(\Aa)$. Note that $\Aa_k$, as a matrix associated with $G^k$,  
is symmetric with entries $(\Aa)_{uv}=1$ when $u=v$ or $uv\not\in E$ (that is, the condition preceding \eqref{theta1}). 
Moreover, it has spectral radius $\rho(\Aa_k)=n-nf_k(\theta_0)+\Psi(f_k)=\Psi(f_k)$ since $f_k(\theta_0)=1$. Thus, the result follows.
\end{proof}

 Taking $k=1$ in Theorem \ref{thm:ratiotypeboundsupperboundslovasztheta}, we obtain an alternative spectral proof of the following known result (see \cite[Theorem 9]{l79}, where Lov\'asz credited Hoffman for this result):

\begin{corollary}
 Let $G$ be a $\delta$-regular graph on $n$ vertices with minimum eigenvalue $\theta_d$. Then, its Lov\'asz theta number satisfies
 \begin{equation}
 \vartheta(G)\le n\frac{-\theta_d}{\delta-\theta_d}.
 \label{Lovasz-Hoffman}
 \end{equation}
\end{corollary}

The following example illustrates that equality in \eqref{Lovasz-Hoffman} holds for several graph classes.
\begin{example}\cite[Corollary 5]{l79}
    The cycle $C_n$ is a $2$-regular graph with maximum eigenvalue $\theta_0=2$ and minimum eigenvalue $\theta_d=-2$ if $n=2d$, and $\theta_d=-2\cos(\pi/n)$ if $n=2d+1$.
    Then, computing the Hoffman bound, we get
    $$
    \vartheta(C_n)=\left\{
    \begin{array}{ll} 
    2 & \mbox{if $n$ is even,}\\
    \frac{n\cos(\pi/n)}{1+\cos(\pi/n)} & \mbox{if $n$ is odd.}
    \end{array}
    \right.
    $$
\end{example}

Furthermore, using Theorem \ref{th:Theta-minor-pol} and the above results, we obtain the following corollary.

\begin{corollary}\label{coro:firstpart}
If $G$ is a $k$-partially walk-regular graph, then the following inequalities hold
\begin{equation}
\label{all-ineq}
\alpha_k(G)=\alpha(G^k)\le \Theta(G^k)\le \vartheta(G^k)\le \tr f_k(\Aa).
\end{equation}    
\end{corollary}

Hence, if the exact value of $\alpha_k(G)$ equals $\tr f_k(\Aa)$, or alternatively, if the ratio-type bound for $k$-walk partially walk-regular graphs is tight (so that both numbers are integers), then all inequalities become equalities, and as a byproduct we obtain the value $\Theta(G^k)=\vartheta(G^k)$. For $k=1$, this was already noted by Lov\'asz \cite{l79}.

 We should also note that, with \eqref{thetaLovasz-fk} and \eqref{all-ineq}, we correct the corresponding two last formulas of Fiol \cite{f20} (where $\alpha(G)$ must be $\alpha(G^k)$, and $\vartheta(G^k)$ and  $\Theta(G^k)$ must be interchanged).

\section{Bounds using the Shannon polynomials}
\label{sec:optimizingboundsalphak}

In \cite{h79}, Haemers answered negatively the three questions posed by Lov\'asz \cite{l79} concerning the Shannon capacity of a graph $G$. More precisely, Haemers obtained a bound for the Shannon capacity in terms of the rank of a given matrix related to $G$. In the past years, Haemers' bound has been extended and strengthened (see Bukh and Cox \cite{bc19}), also in the context of quantum information theory (see Duan, Severini, and Winter \cite{dsw13}, Levene, Paulsen, and Todorov \cite{ltp18}, Gribling and Li \cite{gl20}, and Gao, Gribling, and Li \cite{ggl22}). 

In this section, we generalize the celebrated Haemers' bound using what we call `Shannon polynomials'. This generalization does not follow trivially from Haemers' bound,  in the sense that it depends on a polynomial and an optimization method that cannot be directly derived from the original bound by Haemers. 

With this aim, we say that a matrix $\BB=(b_{ij})$ \textit{fits} a graph $G=(V,E)$, denoted by $\BB\propto G$,  if $b_{ii}\neq 0$ for every $i\in V$, and $b_{ij}=0$ if  $ij\notin E$
(observe that the definition of $\BB$ by Haemers \cite{h79} is more restrictive, but the proof also holds using our definition.)

\begin{theorem}[Haemers minimum rank bound \cite{h79}]
\label{th:Haemersbound}
Let $G$ be a graph, and $\BB$ a matrix that fits it. Then,
the Shannon capacity of $G$ satisfies that
$$
\Theta(G)\leq R(G)=\min_{\textbf{B}\,\propto\, G} \rank(\BB).
$$
\end{theorem}

Next we show an extension of Theorem \ref{th:Haemersbound} when considering the power graph. In particular, we show that the following rank-type bound both an  upper bound for the $k$-independence number of a graph and the Shannon capacity of graph powers.

\begin{theorem}[Rank-type bound]\label{th:Hamererankextension2}
The independence number of the $k$-power $G^k$ of a graph $G$, with adjacency matrix $\Aa=\Aa(G)$,
satisfies 
\begin{equation}
\label{eq:all-ineq}
   \alpha_k(G)=\alpha(G^k)\leq \Theta(G^k)\leq \min_{\stackrel{p\in \Real_k[x]}{p({\mathbf{A}})_{ii}\neq 0\,\, \forall i}} \rank(p(\Aa)). 
\end{equation}
\end{theorem}

 \begin{proof}
    Since $\alpha(G^{\boxtimes n})\geq\alpha(G)^{n}$, the Shannon capacity gives an upper bound for the independence number, $\alpha(G) \leq \Theta(G)$. Thus, it also follows that $\alpha_k(G)=\alpha(G^k)\leq \Theta(G^k)$, proving the first inequality in \eqref{eq:all-ineq}.     
    To show the second inequality in \eqref{eq:all-ineq}, we apply Theorem \ref{th:Haemersbound} in the $k$-th power graph. In order to apply Theorem \ref{th:Haemersbound} to the $k$-th power graph $G^k$, we need a matrix $\BB$ that \textit{fits} it; that is, with all non-zero diagonal entries and $b_{ij}=0$ if $\dist(i,j)>k$ in $G$. Moreover, we want $\BB$ to have a rank as small as possible to obtain a good bound. Consequently, following this strategy, we are searching for a polynomial $p_k\in \mathbb{R}_k[x]$ such that $p(\Aa)$ fits $G^k$ and minimizes (or, at least, gives a small value of) $\rank p(\Aa)$. Note that, since $\dgr p_k=k$ (where $\dgr$ is the degree of the polynomial), we have that $p_k(\Aa)_{i,j}=0$ for every pair of vertices $i,j$ in $G$ at a distance $\dist(i,j)>k$ (that is, nonadjacent in $G^k$). Then, $p_k(\Aa)$ fits $G^k$ if and only if $p_k(\Aa)_{ii}\neq 0$ for every $i\in V$. 
     \end{proof}

We call the last inequality in \eqref{eq:all-ineq} the \textit{rank-type bound}.
Note that, for $p(x)=x$, we obtain Theorem \ref{th:Haemersbound}.

Next, we show how to optimize the bound from Theorem \ref{th:Hamererankextension2} in order to find the best/optimal polynomial $p$.
We first recall that, since $\dgr p=k$, $p(\Aa)$ fits $G^k$ if and only if $(p(\Aa))_{ii}\neq 0$ for all $i\in V$. This leads us to propose the following procedure. Since, as mentioned, the Shannon capacity is very hard to compute in general (see, for instance, Alon and Lubetzky \cite{al06}),  we focus on investigating when the inequalities in Theorem \ref{th:Hamererankextension2} are actually equalities.

If $s\in \Real_k[x]$  is a polynomial that satisfies $(s(\Aa))_{ii}\neq 0$ for every $i\in V$, the best result is obtained by what we call the \textit{Shannon polynomials} $s_k(x)$, which are defined as follows. 
Let $G$ be a graph with diameter $D$ and spectrum $\spec G=\{\theta_0,\theta_1^{m_1},\ldots ,\theta_d^{m_d}\}$ with $\theta_0>\theta_1^{m_1}>\cdots >\theta_d^{m_d}$. Let $\b=(b_0,\ldots,b_d)\in \{0,1\}^{d+1}$ and $\m=(m_0,\ldots,m_d)$. For a given $k<D\ (\le d)$, let $I\subset [0,d]$ with $|I|=k$. Then, the Shannon polynomial $s_k$ is the one with zeros at $\theta_i$ for $i\in I$, that is, $s_k(x)=\gamma\prod_{i\in I}(x-\theta_i)$ for some constant $\gamma$, such that
\begin{equation}
\boxed{
 \begin{array}{rl}
 & \\
 {\tt maximize}\  & \displaystyle\sum_{i\in I} m_i, 
 \\
 {\tt subject\ to \ } & I\subset [0,d],\\
                      & \displaystyle\sum_{i=0}^d m_i s_k(\theta_i) =1. 
 \\
 & \\
 \end{array}
}
\label{MILP-Shannon-poly}
 \end{equation}
The idea of this formulation is to find the $k$ eigenvalues that are zeros of $s_k(\Aa)$ that maximize \eqref{MILP-Shannon-poly}, that is, to minimize the number of eigenvalues that are not zeros of $s_k(\Aa)$, which is equal to  $\rank s_k(\Aa)$. 
Hence, we optimize the corresponding bound  $\Theta (G^k)\le n- \sum_{i\in I} m_i$,
provided that the main diagonal of $s_k(\Aa)$ is non-zero. 
Moreover, a simple case in which this last requirement is fulfilled is when $G$ is a $k$-partially walk-regular graph. Then, $(s_k(\Aa))_{ii}=\frac{1}{n}\tr s_k(\Aa)$, which is non-zero because of the condition in \eqref{MILP-Shannon-poly} (where the `1' is for the sake of normalizing). 
A very simple algorithm to find the Shannon polynomial $s_k(x)$
is as follows:

\begin{algorithm}[H]
\caption{Shannon polynomial $s_k(x)$}
\label{alg:cycle-tree}
\begin{algorithmic}[1]
\item[{\bf Step 1}]
Given the spectrum $\spec G=\{\theta_0,\theta_1^{m_1},\ldots ,\theta_d^{m_d}\}$ with $\theta_0>\theta_1^{m_1}>\cdots >\theta_d^{m_d}$, list in non-increasing order the multiplicities $m_{i_0}\ge m_{i_1}\ge \cdots \ge m_{i_d}$ (if $m_{i_j}=m_{i_j+1}$, assume that $\theta_{i_j}>\theta_{i_j+1}$).
\item[{\bf Step 2}] 
Check if the polynomial $p_k=\prod_{j=0}^{k-1}(x-\theta_{i_j})$ 
satisfies $T=\sum_{i=0}^d m_i p_k(\theta_i)\neq 0$.
\item[{\bf Step 3}] 
If $T\neq 0$, let $s_k(x)=\frac{1}{T}p_k(x)$.
\item[{\bf Step 4}] 
Otherwise, if $T=0$, go to Step 1 with $p_k$ obtained by changing $\theta_{i_{k-1}}$ by $\theta_{i_{k}}$.
\item[{\bf Step 5}] 
Repeat the procedure, each time reducing (as much as possible) the sum of the $k$ multiplicities involved until you obtain $T\neq 0$ and stop in Step 3.
\end{algorithmic}
\end{algorithm}

By using the Shannon polynomials, we obtain the following result.

\begin{theorem}\label{thm:secondboundShannonpolys}
Let $G$ be a $k$-partially walk-regular graph with adjacency matrix $\Aa$, spectrum $\spec G=\{\theta_0,\theta_1^{m_1},\ldots ,\theta_d^{m_d}\}$, with $\theta_0>\theta_1^{m_1}>\cdots >\theta_d^{m_d}$, 
and Shannon polynomial $s_k\in \Real_k[x]$. Then, the Shannon capacity of the $k$-power graph $G^k$ satisfies
\begin{equation}
\Theta(G^k)\le \rank s_k(\Aa) = |\{m_i: s_k(\theta_i)\neq 0\}|. 
\end{equation}
\end{theorem}

\begin{example}[Haemers \cite{h79}]
Let $G$ be the complement of the Schl\"afli graph (see, for instance, Seidel \cite{s68}). Then, $G$ is a strongly regular graph with adjacency matrix $\Aa$,  parameters $(n,k;a,c)=(27,10; 1,5)$ and spectrum $\spec G=\{10,1^{20},-5^{6}\}$. Then, the polynomial used by Haemers \cite{h79} corresponds, up to a multiplicative constant, to the Shannon polynomial
$s_1(x)=\frac{1}{20}(1-x)$. Then, since $s_1(1)=0$, we infer that $\Theta(G)\le \rank s_1(\Aa)=27-20=7.$
\end{example}

\begin{example}
Let $G$ be the first subconstituent of McLaughlin graph \cite{ml69}. This is a strongly regular graph with parameters $(n,k;a,c)=(112,30; 2,10)$ and spectrum $\spec G=\{30,2^{90},-10^{21}\}$. Then, the 1-Shannon polynomial is
$s_1(x)=\frac{1}{28}(x-2)$. Then, since $s_1(2)=0$, we infer that $\Theta(G)\le \rank s_1(\Aa)=112-90=22$, whereas the ratio bound gives $\frac{112}{1-\frac{30}{-10}}=28$. Thus, $\Theta(G)\le 22$.
\end{example}

Haemers investigated the general case of a strongly regular graph \cite{H1978}. Recall that the spectrum $\{k,\theta^{m(\theta)},\tau^{m(\tau)}\}$ of such a graph can be deduced from its parameters $(n,k;a,c)$ as follows.
Let $\Delta=(a-c)^2+4(k-c)$. Then,
\begin{align*}
\theta &=(a-c+\sqrt{\Delta})/2,\qquad
\tau =(a-c+\sqrt{\Delta})/2,\\
m(\theta)&=\frac{(n-1)\tau+k}{\tau-\theta}=\frac{1}{2}\left(n-1-\frac{2k+(n-1)(a-c)}{\sqrt{\Delta}}\right),\\
m(\tau)&=\frac{(n-1)\theta+k}{\theta-\tau}=\frac{1}{2}\left(n-1+\frac{2k+(n-1)(a-c)}{\sqrt{\Delta}}\right).\\
\end{align*}
Thus, Haemers \cite{H1978} implicitly used Shannon polynomials of degree 1 to show that
$$
R(G)\le \min \{1+m(\theta),1+m(\tau)\}.
$$
Moreover, from the results of Schrijver \cite{sh79}, he also noted that if $G$ is a (non-trivial) strongly regular graph, we have equality in \eqref{Lovasz-Hoffman}, that is
$\theta(G)=H_1=\frac{n}{1-\frac{k}{\tau}}$.
Haemers also commented that it is not difficult to find strongly regular graphs satisfying $1+m(\tau)<H_1=\frac{n}{1-\frac{k}{tau}}$. In fact, he provided an infinite family derived from an elliptic quadric in $PG(5,q)$, with parameters $n=(q^3+1)(q+1)$, $k=q(q^2+1)$, $\theta=q-1$, $\tau=-q^2-1$, and $m(\tau)=q(q^2-q+1)$, so
$$
R(G)\le 1+q(q^2-q+1)<q^3+1=H_1=\theta(G).
$$

\begin{figure}[htp!]
	\begin{center}
 \includegraphics[width=6cm]{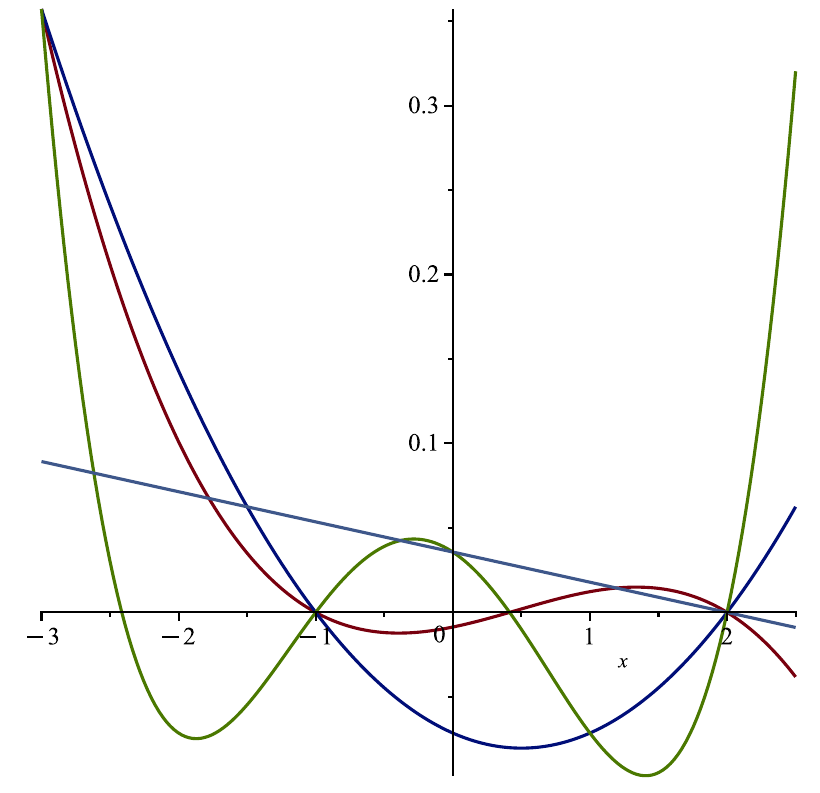}
	\end{center}
	\vskip-.25cm
	\caption{The Shannon polynomials $s_i(x)$ of the Coxeter graph for $i=1,2,3,4$ (Example \ref{ex:coxeter-shannon}).}
	\label{fig:Shannon-pols-Coxeter}
\end{figure}

\begin{example}
\label{ex:coxeter-shannon}
The Coxeter graph $C$ is a cubic distance-regular graph with $n=28$ vertices, diameter $d=4$, and spectrum $\spec C=\{3^1,2^8,(\sqrt{2}-1)^6,-1^7, (-\sqrt{2}-1)^6\}$ (see  Examples \ref{ex:f3+g3} and \ref{ex:coxeter2}). Then, the ordered multiplicities are $8,7,6,6,1$, and the Shannon polynomials turn out to be (see Figure \ref{fig:Shannon-pols-Coxeter}):
\begin{align*}
s_1(x) &=\frac{1}{-56}(x-2),\\
s_2(x) &= \frac{1}{28}(x^2-x-2),\\
s_3(x) &=\frac{1}{-56-28\sqrt{2}}(x^3-\sqrt{2}x^2
-[3-\sqrt{2}]x+2\sqrt{2}-2),\\
s_4(x) &=\frac{1}{56}(x^4+x^3 -5x^2-3x+2).
\end{align*}
Hence, using the fact that we can compute the rank by adding up the multiplicities of the eigenvalues where the Shannon polynomials do not vanish, we obtain
$$
\Theta(C)\le 20,\quad\Theta(C^2)\le 13,\quad \Theta(C^3)\le 7, \quad \Theta(C^4)=1.
$$
\end{example}

\section{Concluding remarks}

In this paper, we presented two algebraic bounds on the Shannon capacity of graph powers: the ratio-type bound and the rank-type bound. In particular, we used the so-called minor polynomials and an LP introduced by the third author in \cite{f20} to show a first upper bound for the Shannon capacity of the $k$-th power of a graph; the ratio-type bound. We also proved that this first spectral bound provides an alternative spectral proof of the fact that the Hoffman bound on the independence number of a graph is also an upper bound on the Shannon capacity of the graph, which is a result originally shown by Lov\'asz \cite{l79}. Secondly, we proposed an extension of the Haemers minimum rank bound on the Shannon capacity of a graph such that it also holds in the more general case of graph powers. We did so by using a new family of polynomials, the so-called Shannon polynomials, which we showed that can be computed by running a simple algorithm or by solving an LP. While the ratio-type bound on the $k$-independence number of a graph has recently been used very successfully in coding theory by the first author (see e.g. \cite{AKR2024,ANR2025, AAR2025, APR2025}), in this paper we show its first application in the context of information theory.

The Appendix contains an overview of the performance of the two new algebraic bounds on the Shannon capacity of graph powers.
Note that both bounds are sharp for some Sage-named graphs (see Tables \ref{fig:alsoLovaszthetak=2}
- \ref{fig:alsoLovaszthetak=5}) and powers of cycles (see Tables \ref{table:cyclepowersranktypeandratiotypeboundsk=4}-\ref{table:cyclepowersranktypeandratiotypeboundsk=5}). Note that when these two new algebraic bounds are tight, they can be used to easily derive the exact values of the Lov\'asz theta number (which relies on solving an SDP) and the Shannon capacity (which is not known to be computable) of the corresponding graph power. In fact, the computational results show that in several instances the ratio-type bound is as strong as the celebrated Lov\'asz theta bound. However, while the Lov\'asz theta bound is computed using semidefinite programming (SDP) methods, the ratio-type only relies on solving a linear program (LP) - practice shows that solving SDPs often struggles against practical intractability and requires highly specialized methods, while LPs are also solvable in polynomial time but are faster and easier to compute. We also observe that both ratio-type and rank-type bounds seem to be incomparable. For $k=1$, we found some instances of strongly-regular graphs for which the rank-type bound outperforms the ratio-type bound (see Table  \ref{tab:srgk=1}), agreeing with known results by Haemers \cite{H1978}. Also, for $k=2$ we obtained several instances of distance-regular graphs for which the rank-type bound is sharp (see Table \ref{tab:DRGalpha2}), and also for multiple Sagenamed graphs for small $k$. Thus, while the ratio-type bound seems to generally outperform the rank-type bound when $k>1$, the rank-type bound can be applied to irregular random graphs with large diameter, while the ratio-type bound assumes graph regularity and thus is less general. The power of the new bounds on the Shannon capacity of graph powers relies on the combination of spectral/algebraic methods with optimization methods, which seems to often capture nicely the graph structure. 


\subsection*{Acknowledgments} 
A. Abiad is supported by NWO (Dutch Research Council) through the grants VI.Vidi.213.085 and OCENW.KLEIN.475. C. Dalf\'o and M. A. Fiol are funded by AGAUR from the Catalan Government under project 2021SGR00434 and MICINN from the Spanish Government under project PID2020-115442RB-I00.
M. A. Fiol's research is also supported by a grant from the Universitat Polit\`ecnica de Catalunya, reference AGRUPS-2025. The authors thank Luuk Reijnders for his support with Sagemath. The first author thanks Jeroen Zuiddam for inspiring discussions on the topic.


\newpage

\section*{Appendix}


{\tiny{
\begin{table}[!ht]
    \centering
    \tiny
    \begin{tabular}{|llll|cc|}
    \hline
    $n$ & $d$ & $\lambda$ & $\mu$ & Rank-type & Ratio-type \\
     &  &  &  & bound & bound \\
    \hline
    $5$ & $2$ & $0$ & $1$ & $3$ & $2$ \\
    $9$ & $4$ & $1$ & $2$ & $5$ & $3$ \\
    $10$ & $3$ & $0$ & $1$ & $5$ & $4$ \\
    $10$ & $6$ & $3$ & $4$ & $5$ & $2$ \\
    $13$ & $6$ & $2$ & $3$ & $7$ & $3$ \\
    $15$ & $6$ & $1$ & $3$ & $6$ & $5$ \\
    $15$ & $8$ & $4$ & $4$ & $6$ & $3$ \\
    $16$ & $5$ & $0$ & $2$ & $6$ & $6$ \\
    $16$ & $6$ & $2$ & $2$ & $7$ & $4$ \\
    $16$ & $9$ & $4$ & $6$ & $7$ & $4$ \\
    $16$ & $10$ & $6$ & $6$ & $6$ & $2$ \\
    $17$ & $8$ & $3$ & $4$ & $9$ & $4$ \\
    $21$ & $10$ & $3$ & $6$ & $7$ & $6$ \\
    $21$ & $10$ & $5$ & $4$ & $7$ & $3$ \\
    $25$ & $8$ & $3$ & $2$ & $9$ & $5$ \\
    $25$ & $12$ & $5$ & $6$ & $13$ & $5$ \\
    $25$ & $16$ & $9$ & $12$ & $9$ & $5$ \\
    $26$ & $10$ & $3$ & $4$ & $13$ & $6$ \\
    $26$ & $15$ & $8$ & $9$ & $13$ & $4$ \\
    $\mathbf{27}$ & $\mathbf{10}$ & $\mathbf{1}$ & $\mathbf{5}$ & $\mathbf{7}$ & $\mathbf{9}$ \\
    $27$ & $16$ & $10$ & $8$ & $7$ & $3$ \\
    $28$ & $12$ & $6$ & $4$ & $8$ & $4$ \\
    $28$ & $15$ & $6$ & $10$ & $8$ & $7$ \\
    $29$ & $14$ & $6$ & $7$ & $15$ & $5$ \\
    $35$ & $16$ & $6$ & $8$ & $15$ & $7$ \\
    $35$ & $18$ & $9$ & $9$ & $15$ & $5$ \\
    $36$ & $10$ & $4$ & $2$ & $11$ & $6$ \\
    $36$ & $14$ & $4$ & $6$ & $15$ & $8$ \\
    $36$ & $14$ & $7$ & $4$ & $9$ & $4$ \\
    $36$ & $15$ & $6$ & $6$ & $16$ & $6$ \\
    $36$ & $20$ & $10$ & $12$ & $16$ & $6$ \\
    $36$ & $21$ & $10$ & $15$ & $9$ & $8$ \\
    $36$ & $21$ & $12$ & $12$ & $15$ & $4$ \\
    $36$ & $25$ & $16$ & $20$ & $11$ & $6$ \\
    $37$ & $18$ & $8$ & $9$ & $19$ & $6$ \\
    $40$ & $12$ & $2$ & $4$ & $16$ & $10$ \\
    $40$ & $27$ & $18$ & $18$ & $16$ & $4$ \\
    $41$ & $20$ & $9$ & $10$ & $21$ & $6$ \\
    $45$ & $12$ & $3$ & $3$ & $21$ & $9$ \\
    $45$ & $16$ & $8$ & $4$ & $10$ & $5$ \\
    $45$ & $22$ & $10$ & $11$ & $23$ & $6$ \\
    $45$ & $28$ & $15$ & $21$ & $10$ & $9$ \\
    $45$ & $32$ & $22$ & $24$ & $21$ & $5$ \\
    $49$ & $12$ & $5$ & $2$ & $13$ & $7$ \\
    $49$ & $18$ & $7$ & $6$ & $19$ & $7$ \\
    $49$ & $24$ & $11$ & $12$ & $25$ & $7$ \\
    $49$ & $30$ & $17$ & $20$ & $19$ & $7$ \\
    $49$ & $36$ & $25$ & $30$ & $13$ & $7$ \\
    $50$ & $7$ & $0$ & $1$ & $22$ & $15$ \\
    $50$ & $21$ & $8$ & $9$ & $25$ & $8$ \\
    $50$ & $28$ & $15$ & $16$ & $25$ & $6$ \\
    $50$ & $42$ & $35$ & $36$ & $22$ & $3$ \\
    $53$ & $26$ & $12$ & $13$ & $27$ & $7$ \\
    $55$ & $18$ & $9$ & $4$ & $11$ & $5$ \\
    $55$ & $36$ & $21$ & $28$ & $11$ & $10$ \\
    $56$ & $10$ & $0$ & $2$ & $21$ & $16$ \\
    $56$ & $45$ & $36$ & $36$ & $21$ & $3$ \\
    $57$ & $24$ & $11$ & $9$ & $19$ & $6$ \\
    $57$ & $32$ & $16$ & $20$ & $19$ & $9$ \\
    $61$ & $30$ & $14$ & $15$ & $31$ & $7$ \\
    $63$ & $30$ & $13$ & $15$ & $28$ & $9$ \\
    $63$ & $32$ & $16$ & $16$ & $28$ & $7$ \\
    $64$ & $14$ & $6$ & $2$ & $15$ & $8$ \\
    $64$ & $18$ & $2$ & $6$ & $19$ & $16$ \\
    $64$ & $21$ & $8$ & $6$ & $22$ & $8$ \\
    $64$ & $27$ & $10$ & $12$ & $28$ & $10$ \\
    $64$ & $28$ & $12$ & $12$ & $29$ & $8$ \\
    $64$ & $35$ & $18$ & $20$ & $29$ & $8$ \\
    \hline
    \end{tabular}
    \begin{tabular}{|llll|cc|}
    \hline
     $n$ & $d$ & $\lambda$ & $\mu$ & Rank-type & Ratio-type \\
     &  &  &  & bound & bound \\
    \hline
    $64$ & $36$ & $20$ & $20$ & $28$ & $6$ \\
    $64$ & $42$ & $26$ & $30$ & $22$ & $8$ \\
    $64$ & $45$ & $32$ & $30$ & $19$ & $4$ \\
    $64$ & $49$ & $36$ & $42$ & $15$ & $8$ \\
    $65$ & $32$ & $15$ & $16$ & $33$ & $8$ \\
    $66$ & $20$ & $10$ & $4$ & $12$ & $6$ \\
    $66$ & $45$ & $28$ & $36$ & $12$ & $11$ \\
    $70$ & $27$ & $12$ & $9$ & $21$ & $7$ \\
    $70$ & $42$ & $23$ & $28$ & $21$ & $10$ \\
    $73$ & $36$ & $17$ & $18$ & $37$ & $8$ \\
    $77$ & $16$ & $0$ & $4$ & $22$ & $21$ \\
    $77$ & $60$ & $47$ & $45$ & $22$ & $3$ \\
    $78$ & $22$ & $11$ & $4$ & $13$ & $6$ \\
    $78$ & $55$ & $36$ & $45$ & $13$ & $12$ \\
    $81$ & $16$ & $7$ & $2$ & $17$ & $9$ \\
    $81$ & $20$ & $1$ & $6$ & $21$ & $21$ \\
    $81$ & $24$ & $9$ & $6$ & $25$ & $9$ \\
    $81$ & $30$ & $9$ & $12$ & $31$ & $13$ \\
    $81$ & $32$ & $13$ & $12$ & $33$ & $9$ \\
    $81$ & $40$ & $19$ & $20$ & $41$ & $9$ \\
    $81$ & $48$ & $27$ & $30$ & $33$ & $9$ \\
    $81$ & $50$ & $31$ & $30$ & $31$ & $6$ \\
    $81$ & $56$ & $37$ & $42$ & $25$ & $9$ \\
    $81$ & $60$ & $45$ & $42$ & $21$ & $3$ \\
    $81$ & $64$ & $49$ & $56$ & $17$ & $9$ \\
    $82$ & $36$ & $15$ & $16$ & $41$ & $10$ \\
    $82$ & $45$ & $24$ & $25$ & $41$ & $8$ \\
    $85$ & $20$ & $3$ & $5$ & $35$ & $17$ \\
    $85$ & $64$ & $48$ & $48$ & $35$ & $5$ \\
    $89$ & $44$ & $21$ & $22$ & $45$ & $9$ \\
    $91$ & $24$ & $12$ & $4$ & $14$ & $7$ \\
    $91$ & $66$ & $45$ & $55$ & $14$ & $13$ \\
    $96$ & $19$ & $2$ & $4$ & $39$ & $20$ \\
    $96$ & $20$ & $4$ & $4$ & $46$ & $16$ \\
    $96$ & $75$ & $58$ & $60$ & $46$ & $6$ \\
    $96$ & $76$ & $60$ & $60$ & $39$ & $4$ \\
    $97$ & $48$ & $23$ & $24$ & $49$ & $9$ \\
    $99$ & $48$ & $22$ & $24$ & $45$ & $11$ \\
    $99$ & $50$ & $25$ & $25$ & $45$ & $9$ \\
    $100$ & $18$ & $8$ & $2$ & $19$ & $10$ \\
    $100$ & $22$ & $0$ & $6$ & $23$ & $26$ \\
    $100$ & $27$ & $10$ & $6$ & $28$ & $10$ \\
    $100$ & $33$ & $14$ & $9$ & $25$ & $8$ \\
    $100$ & $36$ & $14$ & $12$ & $37$ & $10$ \\
    $100$ & $44$ & $18$ & $20$ & $45$ & $12$ \\
    $100$ & $45$ & $20$ & $20$ & $46$ & $10$ \\
    $100$ & $54$ & $28$ & $30$ & $46$ & $10$ \\
    $100$ & $55$ & $30$ & $30$ & $45$ & $8$ \\
    $100$ & $63$ & $38$ & $42$ & $37$ & $10$ \\
    $100$ & $66$ & $41$ & $48$ & $25$ & $12$ \\
    $100$ & $72$ & $50$ & $56$ & $28$ & $10$ \\
    $100$ & $77$ & $60$ & $56$ & $23$ & $3$ \\
    $100$ & $81$ & $64$ & $72$ & $19$ & $10$ \\
    $101$ & $50$ & $24$ & $25$ & $51$ & $10$ \\
    $105$ & $26$ & $13$ & $4$ & $15$ & $7$ \\
    $105$ & $32$ & $4$ & $12$ & $21$ & $25$ \\
    $105$ & $72$ & $51$ & $45$ & $21$ & $4$ \\
    $105$ & $78$ & $55$ & $66$ & $15$ & $14$ \\
    $109$ & $54$ & $26$ & $27$ & $55$ & $10$ \\
    $111$ & $44$ & $19$ & $16$ & $37$ & $9$ \\
    $111$ & $66$ & $37$ & $42$ & $37$ & $12$ \\
    $\mathbf{112}$ & $\mathbf{30}$ & $\mathbf{2}$ & $\mathbf{10}$ & $\mathbf{22}$ & $\mathbf{28}$ \\
     $112$ & $81$ & $60$ & $54$ & $22$ & $4$ \\
     $113$ & $56$ & $27$ & $28$ & $57$ & $10$ \\
     $117$ & $36$ & $15$ & $9$ & $27$ & $9$ \\
     $117$ & $80$ & $52$ & $60$ & $27$ & $13$ \\
     $119$ & $54$ & $21$ & $27$ & $35$ & $17$ \\
     $119$ & $64$ & $36$ & $32$ & $35$ & $7$ \\
    \hline
    \end{tabular}
    \caption{Comparison of the performance of the two new algebraic bounds for $k=1$ for several strongly regular graphs. 
    Entry is in bold if the rank-type bound performs strictly better.}
    \label{tab:srgk=1}
\end{table}
}}

\begin{table}[!ht]
    \centering
    \begin{tabular}{|l|cc|c|}
    \hline
    Intersection Array & Rank-type  & Ratio-type & $\alpha_2(G)$ \\
    & bound & bound &  \\
    \hline
    $\left[4, 3, 2, 1, 1, 2, 3, 4\right]$ & $6$ & $2$ & $2$ \\
    $\left[8, 6, 1, 1, 3, 8\right]$ & $7$ & $3$ & $3$ \\
    $\left[8, 6, 1, 1, 3, 8\right]$ & $7$ & $3$ & $3$ \\
    $\left[7, 6, 4, 1, 3, 7\right]$ & $\mathbf{2}$ & $2$ & $2$ \\
    $\left[6, 5, 4, 1, 2, 6\right]$ & $\mathbf{2}$ & $2$ & $2$ \\
    $\left[5, 4, 1, 1, 1, 1, 4, 5\right]$ & $14$ & $3$ & $2$ \\
    $\left[4, 3, 3, 1, 1, 2\right]$ & $\mathbf{7}$ & $7$ & $7$ \\
    $\left[6, 4, 2, 1, 1, 1, 4, 6\right]$ & $18$ & $4$ & $3$ \\
    $\left[27, 10, 1, 1, 10, 27\right]$ & $8$ & $2$ & $2$ \\
    $\left[9, 6, 3, 1, 2, 3\right]$ & $10$ & $4$ & $4$ \\
    $\left[4, 3, 3, 2, 2, 1, 1, 1, 1, 2, 2, 3, 3, 4\right]$ & $42$ & $14$ & $14$ \\
    $\left[16, 9, 4, 1, 1, 4, 9, 16\right]$ & $22$ & $3$ & $2$ \\
    $\left[13, 12, 9, 1, 4, 13\right]$ & $\mathbf{2}$ & $2$ & $2$ \\
    $\left[7, 6, 6, 1, 1, 1, 1, 6, 6, 7\right]$ & $44$ & $6$ & $4$ \\
    $\left[15, 14, 10, 3, 1, 5, 12, 15\right]$ & $23$ & $5$ & $4$ \\
    $\left[10, 9, 8, 2, 1, 1, 2, 8, 9, 10\right]$ & $42$ & $7$ & $4$ \\
    $\left[5, 4, 4, 3, 1, 1, 2, 2\right]$ & $36$ & $13$ & $12$ \\
    $\left[20, 12, 6, 2, 1, 4, 9, 16\right]$ & $36$ & $6$ & $3$ \\
    $\left[16, 15, 12, 4, 1, 1, 4, 12, 15, 16\right]$ & $44$ & $7$ & $4$ \\
    \hline
\end{tabular}
    \caption{Comparison of the performance of the two new algebraic bounds for $\alpha_2$ for several small distance-regular graphs. Entry is in bold if the rank-type bound is sharp.}
    \label{tab:DRGalpha2}
\end{table}

\begin{table}[!ht]
    \centering
    \begin{tabular}{|l|cc|c|}
    \hline
    Intersection Array& Rank-type  & Ratio-type & $\alpha_3(G)$ \\
    & bound & bound &  \\
    \hline
    $\left[4, 3, 2, 1, 1, 2, 3, 4\right]$ & $5$ & $2$ & $2$ \\
    $\left[5, 4, 1, 1, 1, 1, 4, 5\right]$ & $9$ & $2$ & $2$ \\
    $\left[6, 4, 2, 1, 1, 1, 4, 6\right]$ & $13$ & $3$ & $3$ \\
    $\left[4, 3, 3, 2, 2, 1, 1, 1, 1, 2, 2, 3, 3, 4\right]$ & $28$ & $7$ & $7$ \\
    $\left[16, 9, 4, 1, 1, 4, 9, 16\right]$ & $8$ & $2$ & $2$ \\
    $\left[7, 6, 6, 1, 1, 1, 1, 6, 6, 7\right]$ & $23$ & $3$ & $2$ \\
    $\left[15, 14, 10, 3, 1, 5, 12, 15\right]$ & $22$ & $3$ & $2$ \\
    $\left[10, 9, 8, 2, 1, 1, 2, 8, 9, 10\right]$ & $22$ & $3$ & $2$ \\
    $\left[5, 4, 4, 3, 1, 1, 2, 2\right]$ & $9$ & $8$ & $7$ \\
    $\left[20, 12, 6, 2, 1, 4, 9, 16\right]$ & $9$ & $2$ & $2$ \\
    $\left[16, 15, 12, 4, 1, 1, 4, 12, 15, 16\right]$ & $23$ & $3$ & $2$ \\
    \hline
\end{tabular}
    \caption{Comparison of the performance of the two new algebraic bounds for $\alpha_3$ for several small distance-regular graphs.}
    \label{tab:DRGalpha3}
\end{table}


{\tiny{
\begin{table}[!ht]
\centering
\renewcommand{\arraystretch}{1.15}
\begin{tabular}{|l|c c c|c|}
\hline
Graph & Rank-type bound & Ratio-type bound & Lov\'asz theta number & $\alpha_2(G)$ \\ \hline
Balaban 10-cage & $54$ & $17$ & $17$ & $17$ \\
Balaban 11-cage & $88$ & $27$ & $26$ & $24$ \\
Bidiakis cube & $6$ & $3$ & $2$ & $2$ \\
Biggs-Smith graph & $67$ & $23$ & $23$ & $21$ \\
Blanusa First Snark Graph & $13$ & $4$ & $4$ & $4$ \\
Blanusa Second Snark Graph & $14$ & $4$ & $4$ & $4$ \\
Brinkmann graph & $17$ & $3$ & $3$ & $3$ \\
Bucky Ball & $46$ & $14$ & $12$ & $12$ \\
Chvatal graph & $6$ & $2$ & $1$ & $1$ \\
Clebsch graph & $1$ & $1$ & $1$ & $1$ \\
Conway-Smith graph for 3S7 & $19$ & $4$ & $4$ & $3$ \\
Coxeter Graph & $13$ & $7$ & $7$ & $7$ \\
Desargues Graph & $10$ & $5$ & $5$ & $4$ \\
Dejter Graph & $58$ & $16$ & $16$ & $16$ \\
Dodecahedron & $11$ & $4$ & $4$ & $4$ \\
Double star snark & $22$ & $7$ & $7$ & $6$ \\
Dyck graph & $14$ & $8$ & $8$ & $8$ \\
F26A Graph & $14$ & $6$ & $6$ & $6$ \\
Flower Snark & $16$ & $5$ & $5$ & $5$ \\
Folkman Graph & $6$ & $3$ & $3$ & $3$ \\
Foster Graph & $54$ & $22$ & $22$ & $21$ \\
Foster graph for 3.Sym(6) graph & $18$ & $4$ & $4$ & $3$ \\
Franklin graph & $6$ & $3$ & $2$ & $2$ \\
Gray graph & $26$ & $11$ & $11$ & $11$ \\
Harries Graph & $60$ & $17$ & $17$ & $17$ \\
Harries-Wong graph & $60$ & $17$ & $17$ & $17$ \\
Heawood graph & $2$ & $2$ & $2$ & $2$ \\
Hexahedron & $2$ & $2$ & $2$ & $2$ \\
Hoffman Graph & $6$ & $2$ & $2$ & $2$ \\
Hoffman-Singleton graph & $1$ & $1$ & $1$ & $1$ \\
Holt graph & $15$ & $4$ & $3$ & $3$ \\
Icosahedron & $4$ & $2$ & $2$ & $2$ \\
Klein 3-regular Graph & $40$ & $13$ & $13$ & $12$ \\
Klein 7-regular Graph & $9$ & $3$ & $3$ & $3$ \\
Ljubljana graph & $84$ & $27$ & $27$ & $26$ \\
McGee graph & $16$ & $6$ & $5$ & $5$ \\
Meredith Graph & $38$ & $14$ & $10$ & $10$ \\
Moebius-Kantor Graph & $9$ & $4$ & $4$ & $4$ \\
Nauru Graph & $12$ & $6$ & $6$ & $6$ \\
Octahedron & $1$ & $1$ & $1$ & $1$ \\
Pappus Graph & $8$ & $3$ & $3$ & $3$ \\
Perkel Graph & $19$ & $5$ & $5$ & $5$ \\
Petersen graph & $1$ & $1$ & $1$ & $1$ \\
Robertson Graph & $15$ & $3$ & $3$ & $3$ \\
Schläfli graph & $1$ & $1$ & $1$ & $1$ \\
Shrikhande graph & $1$ & $1$ & $1$ & $1$ \\
Sims-Gewirtz Graph & $1$ & $1$ & $1$ & $1$ \\
Sylvester Graph & $10$ & $6$ & $6$ & $6$ \\
Szekeres Snark Graph & $33$ & $12$ & $10$ & $9$ \\
Thomsen graph & $1$ & $1$ & $1$ & $1$ \\
Truncated Tetrahedron & $6$ & $3$ & $3$ & $3$ \\
Tutte-Coxeter graph & $11$ & $6$ & $6$ & $6$ \\
Twinplex Graph & $8$ & $2$ & $2$ & $2$ \\
Wagner Graph & $4$ & $2$ & $1$ & $1$ \\
Wells graph & $14$ & $3$ & $3$ & $2$ \\
\hline
\end{tabular}
\caption{The two new algebraic bounds in comparison with the Lov\'asz theta number for Sagemath named graphs when $k=2$.}
 	\label{fig:alsoLovaszthetak=2}
\end{table}
}}

{\footnotesize{
\begin{table}[!ht]
\centering
\renewcommand{\arraystretch}{1.15}
\begin{tabular}{|l|c c c|c|}
\hline
Graph & Rank-type bound & Ratio-type bound & Lov\'asz theta number & $\alpha_3(G)$ \\
\hline
Balaban 10-cage & $46$ & $11$ & $10$ & $9$ \\
Biggs-Smith graph & $51$ & $14$ & $14$ & $12$ \\
Blanusa First Snark Graph & $11$ & $2$ & $2$ & $2$ \\
Blanusa Second Snark Graph & $13$ & $2$ & $2$ & $2$ \\
Brinkmann graph & $15$ & $1$ & $1$ & $1$ \\
Bucky Ball & $41$ & $8$ & $7$ & $7$ \\
Conway-Smith graph for 3S7 & $13$ & $3$ & $3$ & $3$ \\
Coxeter Graph & $7$ & $4$ & $4$ & $4$ \\
Desargues Graph & $6$ & $2$ & $2$ & $2$ \\
Dejter Graph & $44$ & $8$ & $8$ & $8$ \\
Dodecahedron & $7$ & $2$ & $2$ & $2$ \\
Double star snark & $18$ & $4$ & $4$ & $4$ \\
Dyck graph & $8$ & $4$ & $4$ & $4$ \\
F26A Graph & $8$ & $3$ & $3$ & $3$ \\
Flower Snark & $14$ & $2$ & $2$ & $2$ \\
Folkman Graph & $5$ & $2$ & $2$ & $2$ \\
Foster Graph & $42$ & $15$ & $15$ & $15$ \\
Foster graph for 3.Sym(6) graph & $13$ & $3$ & $3$ & $3$ \\
Franklin graph & $4$ & $1$ & $1$ & $1$ \\
Gray graph & $20$ & $9$ & $9$ & $9$ \\
Harries Graph & $55$ & $10$ & $10$ & $10$ \\
Harries-Wong graph & $55$ & $10$ & $10$ & $9$ \\
Heawood graph & $1$ & $1$ & $1$ & $1$ \\
Hexahedron & $1$ & $1$ & $1$ & $1$ \\
Hoffman Graph & $5$ & $2$ & $2$ & $2$ \\
Holt graph & $9$ & $2$ & $1$ & $1$ \\
Icosahedron & $1$ & $1$ & $1$ & $1$ \\
Klein 3-regular Graph & $33$ & $7$ & $7$ & $7$ \\
Klein 7-regular Graph & $1$ & $1$ & $1$ & $1$ \\
Ljubljana graph & $70$ & $18$ & $18$ & $17$ \\
McGee graph & $12$ & $3$ & $2$ & $2$ \\
Moebius-Kantor Graph & $6$ & $2$ & $2$ & $2$ \\
Nauru Graph & $9$ & $4$ & $4$ & $4$ \\
Pappus Graph & $7$ & $3$ & $3$ & $3$ \\
Perkel Graph & $1$ & $1$ & $1$ & $1$ \\
Robertson Graph & $13$ & $1$ & $1$ & $1$ \\
Sylvester Graph & $1$ & $1$ & $1$ & $1$ \\
Szekeres Snark Graph & $29$ & $7$ & $6$ & $6$ \\
Truncated Tetrahedron & $3$ & $1$ & $1$ & $1$ \\
Tutte-Coxeter graph & $10$ & $5$ & $5$ & $5$ \\
Twinplex Graph & $6$ & $1$ & $1$ & $1$ \\
Wagner Graph & $2$ & $1$ & $1$ & $1$ \\
Wells graph & $9$ & $2$ & $2$ & $2$ \\
\hline
\end{tabular}
\caption{The two new algebraic bounds in comparison with the Lov\'asz theta number for Sagemath named graphs when $k=3$.}
 	\label{fig:alsoLovaszthetak=3}
\end{table}
}}

{\small{
\begin{table}[!ht]
\centering
\renewcommand{\arraystretch}{1.15}
\begin{tabular}{|l|c c c|c|}
\hline
Graph & Rank-type bound & Ratio-type bound & Lov\'asz theta number & $\alpha_4(G)$ \\
\hline
Balaban 10-cage & $38$ & $7$ & $5$ & $5$ \\
Balaban 11-cage & $68$ & $10$ & $9$ & $9$ \\
Biggs-Smith graph & $35$ & $7$ & $7$ & $5$ \\
Bucky Ball & $36$ & $6$ & $6$ & $6$ \\
Conway-Smith graph for 3S7 & $1$ & $1$ & $1$ & $1$ \\
Coxeter Graph & $1$ & $1$ & $1$ & $1$ \\
Desargues Graph & $2$ & $2$ & $2$ & $2$ \\
Dejter Graph & $30$ & $4$ & $3$ & $2$ \\
Dodecahedron & $4$ & $2$ & $2$ & $2$ \\
Double star snark & $14$ & $2$ & $1$ & $1$ \\
Dyck graph & $2$ & $2$ & $2$ & $2$ \\
F26A Graph & $2$ & $2$ & $2$ & $2$ \\
Folkman Graph & $1$ & $1$ & $1$ & $1$ \\
Foster Graph & $32$ & $9$ & $9$ & $5$ \\
Foster graph for 3.Sym(6) & $1$ & $1$ & $1$ & $1$ \\
Franklin graph & $2$ & $1$ & $1$ & $1$ \\
Gray graph & $8$ & $3$ & $3$ & $3$ \\
Harries Graph & $50$ & $6$ & $5$ & $5$ \\
Harries-Wong graph & $50$ & $6$ & $5$ & $5$ \\
Hoffman Graph & $1$ & $1$ & $1$ & $1$ \\
Holt graph & $5$ & $1$ & $1$ & $1$ \\
Klein 3-regular Graph & $26$ & $4$ & $4$ & $4$ \\
Ljubljana graph & $56$ & $10$ & $10$ & $8$ \\
McGee graph & $9$ & $1$ & $1$ & $1$ \\
Moebius-Kantor Graph & $2$ & $1$ & $1$ & $1$ \\
Nauru Graph & $5$ & $2$ & $1$ & $1$ \\
Pappus Graph & $1$ & $1$ & $1$ & $1$ \\
Truncated Tetrahedron & $1$ & $1$ & $1$ & $1$ \\
Tutte 12-Cage & $23$ & $9$ & $9$ & $9$ \\
Tutte-Coxeter graph & $1$ & $1$ & $1$ & $1$ \\
Wagner Graph & $1$ & $1$ & $1$ & $1$ \\
Wells graph & $1$ & $1$ & $1$ & $1$ \\
\hline
\end{tabular}
\caption{The two new algebraic bounds in comparison with the Lov\'asz theta number for Sagemath named graphs when $k=4$.}
 	\label{fig:alsoLovaszthetak=4}
\end{table}
}}

{\small{
\begin{table}[!ht]
\centering
\renewcommand{\arraystretch}{1.15}
\begin{tabular}{|l|c c c|c|}
\hline
Graph & Rank-type bound & Ratio-type bound & Lov\'asz theta number & $\alpha_5(G)$ \\
\hline
Balaban 10-cage & $30$ & $5$ & $5$ & $5$ \\
Balaban 11-cage & $60$ & $7$ & $5$ & $4$ \\
Biggs-Smith graph & $19$ & $4$ & $4$ & $4$ \\
Bucky Ball & $31$ & $4$ & $3$ & $3$ \\
Desargues Graph & $1$ & $1$ & $1$ & $1$ \\
Dejter Graph & $23$ & $2$ & $2$ & $2$ \\
Dodecahedron & $1$ & $1$ & $1$ & $1$ \\
Double star snark & $10$ & $1$ & $1$ & $1$ \\
Dyck graph & $1$ & $1$ & $1$ & $1$ \\
F26A Graph & $1$ & $1$ & $1$ & $1$ \\
Foster Graph & $23$ & $4$ & $4$ & $3$ \\
Franklin graph & $1$ & $1$ & $1$ & $1$ \\
Gray graph & $7$ & $3$ & $3$ & $3$ \\
Harries Graph & $45$ & $5$ & $5$ & $5$ \\
Harries-Wong graph & $45$ & $5$ & $5$ & $5$ \\
Holt graph & $1$ & $1$ & $1$ & $1$ \\
Klein 3-regular Graph & $19$ & $2$ & $2$ & $2$ \\
Ljubljana graph & $49$ & $8$ & $8$ & $8$ \\
McGee graph & $6$ & $1$ & $1$ & $1$ \\
Moebius-Kantor Graph & $1$ & $1$ & $1$ & $1$ \\
Nauru Graph & $4$ & $1$ & $1$ & $1$ \\
Tutte 12-Cage & $22$ & $9$ & $9$ & $9$ \\
\hline
\end{tabular}
\caption{The two new algebraic bounds in comparison with the Lov\'asz theta number for Sagemath named graphs when $k=5$.}
 	\label{fig:alsoLovaszthetak=5}
\end{table}
}}


\begin{table}[!ht]
    \centering
    \small
    \begin{tabular}{|c|ccc|c|}
    \hline
    Cycle length & Rank-type bound & Ratio-type bound & Lov\'asz theta number & $\alpha_4(G)$ \\
    \hline
    $8$ & $1$ & $1$ & $1$ & $1$  \\
    $9$ & $1$ & $1$ & $1$ & $1$  \\
    $10$ & $2$ & $2$ & $2$ & $2$  \\
    $11$ & $3$ & $2$ & $2$ & $2$  \\
    $12$ & $4$ & $2$ & $2$ & $2$  \\
    $13$ & $5$ & $2$ & $2$ & $2$  \\
    $14$ & $6$ & $2$ & $2$ & $2$  \\
    $15$ & $7$ & $3$ & $3$ & $3$  \\
    $16$ & $8$ & $3$ & $3$ & $3$  \\
    $17$ & $9$ & $3$ & $3$ & $3$  \\
    $18$ & $10$ & $3$ & $3$ & $3$  \\
    $19$ & $11$ & $3$ & $3$ & $3$  \\
    \hline
    \end{tabular}
    \caption{The two new algebraic bounds for powers of cycles compared with the Lov\'asz
theta number when $k = 4$.}
    \label{table:cyclepowersranktypeandratiotypeboundsk=4}
    \end{table}


 
\begin{table}[!ht]
    \centering
    \small
    \begin{tabular}{|c|ccc|c|}
    \hline
    Cycle length & Rank-type bound & Ratio-type bound & Lov\'asz theta number & $\alpha_5(G)$ \\
    \hline
    $10$ & $1$ & $1$ & $1$ & $1$  \\
    $11$ & $1$ & $1$ & $1$ & $1$  \\
    $12$ & $3$ & $2$ & $2$ & $2$  \\
    $13$ & $3$ & $2$ & $2$ & $2$  \\
    $14$ & $4$ & $2$ & $2$ & $2$  \\
    $15$ & $5$ & $2$ & $2$ & $2$  \\
    $16$ & $6$ & $2$ & $2$ & $2$  \\
    $17$ & $7$ & $2$ & $2$ & $2$  \\
    $18$ & $8$ & $3$ & $3$ & $3$  \\
    $19$ & $9$ & $3$ & $3$ & $3$  \\
    \hline
    \end{tabular}
    \caption{The two new algebraic bounds for powers of cycles compared with the Lov\'asz
theta number when $k = 5$.}
\label{table:cyclepowersranktypeandratiotypeboundsk=5}
    \end{table}


\begin{thebibliography}{99}

\bibitem{AAR2025} A. Abiad, G. Alfarano, and A. Ravagnani,
Eigenvalue bounds and alternating rank-metric codes, 
\textit{J. Algebra Appl.}, to appear (2025). 
	
\bibitem{acf19}
A. Abiad, G. Coutinho, and M. A. Fiol, 
On the $k$-independence number of graphs, 
\textit{Discrete Math.} {\bf 342} (2019) 2875--2885.

\bibitem{acfnz21}
A. Abiad, G. Coutinho, M. A. Fiol, B. D. Nogueira, and S. Zeijlemaker,
Optimization of eigenvalue bounds for the independence and chromatic number of graph powers,
\textit{Discrete Math.} \textbf{345(3)} (2022) 112706.

\bibitem{AKR2024} A. Abiad, A. P. Khramova, and A. Ravagnani, 
Eigenvalue bounds for sum-rank-metric codes,
\textit{IEEE Trans. Inform. Theory} \textbf{70(7)} (2024) 4843--48554.

\bibitem{ANR2025} A. Abiad, A. Neri, and L. Reijnders,
Eigenvalue bounds for the distance-$t$ chromatic number of a graph and their application to perfect Lee codes, 
\textit{J. Algebra Appl.} (2025), to appear.

\bibitem{APR2025}  A. Abiad, L. Peters, A. Ravagnani, 
The Eigenvalue Method in Coding Theory, 
\textit{Canadian Journal of Mathematics}, to appear (2025).

\bibitem{AZ2024} A. Abiad and J. Zhou,
Unified bounds for the independence number of graph powers, 
arXiv:2411.09450.


\bibitem{a98}
N. Alon,
The Shannon capacity of a union, 
\textit{Combinatorica} {\bf 18} (1998) 301--310. 

\bibitem{Apowers} 
N. Alon, 
Graph powers, 
\url{https://web.math.princeton.edu/~nalon/PDFS/cap2.pdf}.

\bibitem{al06}
N. Alon and E. Lubetzky, 
The Shannon capacity of a graph and the independence numbers of its powers, 
\textit{IEEE Trans. Inform. Theory} {\bf 52(5)} (2006) 2172--2176.

\bibitem{AM2002} 
N. Alon and B. Mohar, 
The chromatic number of graph powers, 
\textit{Comb. Probab. Comput.} {\bf 11} (1993) 1--10.

\bibitem{2} 
T. Andreae and H.-J. Bandelt, 
Performance guarantees for approximation algorithms depending on parametrized triangle inequalities, 
\textit{SIAM J. Discrete Math.} \textbf{8} (1995) 1--16.

\bibitem{AF2004} 
G. Atkinson and A. Frieze,
On the $b$-Independence Number of Sparse Random Graphs, 
\textit{Combin. Probab. Comput.} \textbf{13(3)} (2004) 295--309.

\bibitem{bscrr14} 
M. Basavaraju, L. Sunil Chandran, D. Rajendraprasad, and A. Ramaswamy,
Rainbow connection number of graph power and graph products, 
\textit{Graphs Combin.} \textbf{30(6)} (2014) 1363--1382.

\bibitem{bc00} 
M. A. Bender and C. Chekuri, 
Performance guarantees for the TSP with a parametrized triangle inequality, 
\textit{Inf. Process. Lett.} {\bf 73} (2000) 17--21.

\bibitem{biggs}
N. L. Biggs,
\textit{Algebraic Graph Theory}, 
2nd ed., Cambridge University Press, Cambridge, 1993.

\bibitem{BBZ2024} 
D. de Boer, P. Buys, and J. Zuiddam, 
The asymptotic spectrum distance, graph limits, and the Shannon capacity, \texttt{arXiv:2404.16763}.

\bibitem{b03} 
T. Bohman, 
A limit theorem for the Shannon capacity of odd cycles I, 
\textit{Proc. Amer. Math. Soc.} {\bf 131(11)} (2003) 3559–3569.

\bibitem{b05} 
T. Bohman, 
A limit theorem for the Shannon capacity of odd cycles II, 
\textit{Proc. Amer. Math. Soc.} {\bf 133(2)} (2005) 537--543.


\bibitem{bc19} 
B. Bukh and C. Cox, 
On a Fractional Version of Haemers' Bound, 
\textit{IEEE Trans. Inform. Theory} {\bf 65(6)} (2019) 3340--3348.

\bibitem{BPZ2025} 
P. Buys, S. Polak, and J. Zuiddam, A group-theoretic approach to Shannon capacity of graphs and a limit theorem from lattice packings, 
arXiv:2506.14654.

\bibitem{dg13} 
K. C. Das and J. M. Guo,  
Laplacian eigenvalues of the second power of a graph,  
\textit{Discrete Math.} {\bf 313} (2013) 626--634.

\bibitem{Devos2013AveragePowers}
M. DeVos, J. McDonald, and D. Scheide, 
Average degree in graph powers,
\textit{J. Graph Theory} \textbf{72(1)} (2013) 7--18.

\bibitem{dsw13} 
R. Duan, S. Severini, and A. Winter, 
Zero-Error Communication via Quantum Channels, Noncommutative Graphs, and a Quantum Lov\'asz Number, 
\textit{IEEE Trans. Inform. Theory} \textbf{59(2)} (2013) 1164--1174.


\bibitem{F1999} 
M. A. Fiol, 
Eigenvalue interlacing and weight parameters of graphs, 
\textit{Linear Algebra Appl.} \textbf{290} (1999) 275–301.

\bibitem{f20}
M. A. Fiol,
A new class of polynomials from the spectrum of a graph, and its application to bound the $k$-independence number, 
\textit{Linear Algebra Appl.} \textbf{605} (2020) 1--20.

\bibitem{FH1997} 
P. Firby and J. Haviland, 
Independence and average distance in graphs, 
\textit{Discrete Appl. Math.} \textbf{75} (1997) 27--37.

\bibitem{ggl22}
L. Gao, S. Gribling, and Y. Li, 
On a tracial version of Haemers bound, 
\textit{EEE Trans. Info. Theor.} \textbf{68 (10)} (2022) 6585--6604.

\bibitem{g93}
C. D. Godsil, 
\textit{Algebraic Combinatorics},
Chapman and Hall, London/New York, 1993.

\bibitem{gm80}
C. D. Godsil and B. D. Mckay, 
Feasibility conditions for the existence of walk-regular graphs, 
\textit{Linear Algebra Appl.} {\bf 30} (1980) 51--61.

\bibitem{gl20} 
S. Gribling and Y. Li, 
The Haemers Bound of Noncommutative Graphs, 
\textit{IEEE J. Sel. Areas Inform. Theory} \textbf{1(2)} (2020) 424--431.

\bibitem{H1978}
W. H. Haemers, 
An upper bound for the Shannon capacity of a graph. In L. Lova\'asz, \& V. Sos (Eds.), \textit{Algebraic Methods in Graph Theory}, Szeged, 1978 (Vol. 25 (1981) 267-272). (Colloquia Mathematica Societatis Janos Bolyai), North-Holland Publishing Company.

\bibitem{h79}
W. H. Haemers, 
On some problems of Lov\'asz concerning the Shannon capacity of a graph, 
\textit{IEEE Trans. Inform. Theory} {\bf 25(2)} (1979) 231--232.

\bibitem{h95}
W. H. Haemers, 
Interlacing eigenvalues and graphs, 
\textit{Linear Algebra Appl.} {\bf 226-228} (1995) 593--616.

\bibitem{13} 
H. Hajiabolhassan, 
On colorings of graph powers, 
\textit{Discrete Math.} {\bf 309} (2009) 4299--4305.


\bibitem{kn23}
L. C. Kavi and M. Newman, 
The optimal bound on the $3$-independence number obtainable from a polynomial-type method,
\textit{Discrete Math.} {\bf 346} (2023) 113471.

\bibitem{15} 
P. E. Kearney and D. G. Corneil, 
Tree powers, 
\textit{J. Algorithms} {\bf 29} (1998) 111--131.


\bibitem{16} 
L. C. Lau and D. G. Corneil,  
Recognizing powers of proper interval, split and chordal graphs,  
\textit{SIAM J. Discrete Math.} {\bf 18} (2004) 83--102.

\bibitem{ltp18} 
R. H. Levene, V. I. Paulsen, and I. G. Todorov, 
Complexity and Capacity Bounds for Quantum Channels, 
\textit{IEEE  Trans. Inform. Theory} \textbf{64(10)} (2018) 6917--6928.

\bibitem{lw2021} 
Z. Li and B. Wu,
The $k$-independence number of $t$-connected graphs,
\textit{Appl. Math. Comput.} \textbf{409} (2021) 126412.

\bibitem{l92} 
N. Linial,  
Locality in distributed graph algorithms,  
\textit{SIAM J. Comput.} {\bf 21} (1992) 193--201.

\bibitem{l79}
L. Lov\'asz, 
On the Shannon capacity of a graph,  
\textit{IEEE  Trans. Inform. Theory} {\bf 25(1)} (1979) 1--7.

\bibitem{powersthesis} 
E. Lubetzky, 
Graph Powers and Related Extremal Problems, 
PhD thesis, Tel Aviv University (2007).

\bibitem{ml69}
J. McLaughlin, 
``A simple group of order 898,128,000'', 
in Brauer, R.; Sah, Chih-han (eds.), Theory of Finite Groups (Symposium, Harvard Univ., Cambridge, Mass., 1968), Benjamin, New York, 1969, pp. 109–-111.

\bibitem{sh79}
A. Schrijver, 
A comparison of the Delsarte and Lov\'asz bounds,
\textit{IEEE Trans. Inform. Theory} {\bf 25} (1979) 425--429.

\bibitem{s68}
J. J. Seidel, 
Strongly regular graphs with $(-1,1,0)$-adjacency matrix having eigenvalue $3$, 
\textit{Linear Algebra Appl.} {\bf 1} (1968) 281--298.

\bibitem{s56}
C. Shannon, 
The zero error capacity of a noisy channel, 
\textit{IRE Trans. Inform. Theory} {\bf 2(3)} (1956) 8--19.

\bibitem{Taoqiu2019SharpGraphs} 
Z. Taoqiu, S. O, and Y. Shi,
Sharp upper bounds on the $k$-independence number in regular graphs,
\textit{Graphs Combin.} \textbf{37} (2021) 393--408.
\end{thebibliography}
\end{document}